\documentclass{article}%
\usepackage{amsmath}
\usepackage{amsfonts}
\usepackage{amssymb}
\usepackage{graphicx}%
\providecommand{\U}[1]{\protect\rule{.1in}{.1in}}
\newtheorem{theorem}{Theorem}

\newtheorem{corollary}[theorem]{Corollary}

\newtheorem{definition}[theorem]{Definition}

\newtheorem{lemma}[theorem]{Lemma}

\newtheorem{remark}[theorem]{Remark}

\newenvironment{proof}[1][Proof]{\noindent\textbf{#1.} }{\ \rule{0.5em}{0.5em}}
\begin{document}

\title{Weak mean random attractors for non-local random and stochastic
reaction-diffusion equations}
\author{R. Caballero$^{1}$, P. Mar\'{\i}n-Rubio$^{2}$ and Jos\'{e} Valero$^{1}$\\$^{1}${\small Centro de Investigaci\'{o}n Operativa, Universidad Miguel
Hern\'{a}ndez de Elche,}\\{\small Avda. Universidad s/n, 03202, Elche (Alicante), Spain}\\{\small E.mail:\ jvalero@gmail.com}\\$^{2}${\small Dpto. Ecuaciones Diferenciales y An\'{a}lisis Num\'{e}rico,}\\{\small Universidad de Sevilla, C/Tarfia, 41012-Sevilla, Spain}\\{\small E.mail: pmr@us.es}}
\date{In love memory of Prof. Mar\'{\i}a Jos\'{e} Garrido-Atienza, colleague and
friend, with deep sorrow}
\maketitle

\begin{abstract}
In this paper, we prove the existence of weak pullback mean random attractors
for a non-local stochastic reaction-diffusion equation with a nonlinear
multiplicative noise. The existence and uniqueness of solutions and weak
pullback mean random attractors is also established for a deterministic
non-local reaction-diffusion equations with random initial data.

\end{abstract}

\bigskip



\section{Introduction}

The theory of global attractors for random dynamical systems in
infinite-dimensional spaces has been developed intensively over the last 20
years. In particular, the theory of pathwise pullback random attractors
\cite{CrauelFlandoli} has been applied succesfully to stochastic equations of
different types (see e.g. \cite{CGSV}, \cite{CLR}, \cite{CLV02}, \cite{CMV},
\cite{CDF}, \cite{FanChen}, \cite{fl-schm}, \cite{Han}, \cite{imk-schm},
\cite{JiaDingGao}, \cite{MRRob}, \cite{Wang09}, \cite{YanLiJi}, \cite{Zhao},
\cite{ZhouYin} among many others). However, this theory has an important
limitation. Namely, it relies on a suitable change of variable which can be
applied only to linear stochastic perturbations. For this reason, several
authors have developed the theory of mean-square random attractors (see
\cite{Gu}, \cite{KloedenLorentz}, \cite{Wang}, \cite{Wang2},
\cite{WangZhuKloeden}, \cite{WuKloeden}), which can be applied to equations
with much more general nonlinear noises.

In this paper we apply the theory of weak mean-square random attractors
developed in \cite{Wang} to the following stochastic non-local
reaction-diffusion problem%
\begin{equation}%
\begin{array}
[c]{l}%
du=(a(\Vert u\Vert_{V}^{2})\Delta u+f(u)+h(t,x))dt+\sigma\left(  u\right)
dw\left(  t\right)  \ \text{in }(\tau,\infty)\times\mathcal{O},\\
u=0\quad\text{on }(\tau,\infty)\times\partial\mathcal{O},\\
u(\tau,x)=u_{\tau}(x)\quad\text{for }x\in\mathcal{O},
\end{array}
\label{EqStoch}%
\end{equation}
where $V=H_{0}^{1}(\mathcal{O})$ and $\mathcal{O}$ is a bounded open set of
$\mathbb{R}^{n}.$ Following \cite{Wang} under suitable conditions on the
domain $\mathcal{O}$ and the functions $f$, $h,\sigma$ and $a$ we prove the
existence of a weak mean-square random attractor in two situations: 1) The
equation is deterministic (that is, $\sigma\left(  u\right)  \equiv0$) but the
initial condition $u_{\tau}$ is a random variable;\ 2) $w(t)$ is a two-sided
scalar Wiener process. The existence of weak mean-square random attractors for
the local problem (that is, $a\equiv1$) was established in \cite{Wang}. These
results were generalized to stochastic reaction-diffusion equations generated
by the $p$-laplacian operator \cite{Wang2}.

Reaction-diffusion equations with non-local diffusion of the type%
\[
u_{t}-a(l(u(t)))\Delta u=f(t,u),
\]
where $l:X\rightarrow\mathbb{R}$ is a suitable functional and $X$ is the phase
space (usually $L^{2}(\mathcal{O})$ or $H_{0}^{1}(\mathcal{O})$), appear in
many applications in Physics, Biology and other sciences (see
\cite{ChipotLovat97}, \cite{ChipotLoval99}, \cite{ChipotValente} and the
references therein). When $l$ is a linear functional of the type%
\[
l(u)=\int_{\Omega}\xi(x)u(x,t)dx,
\]
where $\xi(x)$ is a given function in $L^{2}(\mathcal{O})$, the existence of
properties of global attractors in the autonomous and non-autonomous
situations have been established in \cite{Ahn}, \cite{CaHeMa15}, \cite{d0},
\cite{CaHeMa18}, \cite{CaHeMa18B}. The drawback of this case is that we cannot
obtain a Lypaunov function for the solutions, which makes it difficult to
analyze the fine structure of the attractor. However, if we consider a
functional $l$ of the type%
\[
l(u)=\Vert u\Vert_{V}^{2},
\]
then a Lyapunov functions exists \cite{ChipotValente}. In this situation, the
existence and structure of the global attractor in both the single and
set-valued frameworks have been studied in the papers \cite{CCMRV},
\cite{CMRV}, \cite{CarMor}, \cite{LiCarMor}, \cite{MorVal}.

In equation (\ref{EqStoch}) the functional $l(u)=\Vert u\Vert_{V}^{2}$ helps
us to obtain the existence and uniqueness of solutions, because the operator
$u\mapsto-a(\Vert u\Vert_{V}^{2})\Delta u$ is monotone as a map from $V$ onto
its conjugate space $V^{\ast}$.

This paper is organized as follows. In Section 2, we recall the main results
of the theory of weak pullback mean random attractors developed in
\cite{Wang}. Also, we extend it by proving under suitable assumptions that the
mean random attractor can be characterized using complete trajectories. In
Section 3, we prove the existence of weak pullback mean random attractors for
the deterministic equation (\ref{EqStoch}) with random initial data. Finally,
in Section 4, we obtain the existence of weak pullback mean random attractors
for the stochastic equation (\ref{EqStoch}).

\section{Preliminaries:\ abstract theory of weak pullback mean random
attractors}

In this section we recall the main results of the theory of pullback mean
random attractors developed in \cite{Wang} and add some new ones about the
characterization of the pullback attractor by means of complete trajectories.

\subsection{Weak pullback mean attractors over probability spaces}

In this subsection we provide the theory of weak $\mathcal{D}$-pullback mean
random attractors for a mean random dynamical system over a probability space.

Let $X$ be a Banach space with norm $\left\Vert \text{\textperiodcentered
}\right\Vert _{X}$ and let $\left(  \Omega,\mathcal{F},P\right)  $ be a
probability space. For $p\in(1,\infty)$ llet us consider the Banach space
$L^{p}(\Omega,\mathcal{F};X)$ ($L^{p}(\Omega,X)$ for short) of Bochner
integrable functions $y:\Omega\rightarrow X$ such that%
\[
\int_{\Omega}\left\Vert y\right\Vert _{X}^{2}dP<\infty.
\]
We denote by $\mathcal{D}$ a collection of some families $D=\{D(t)\}_{t\in
\mathbb{R}}$ of non-empty bounded sets $D(t)\subset L^{p}(\Omega,X)$:%
\[
\mathcal{D}=\{D=\{D(t)\}_{t\in\mathbb{R}}:D(t)\in\beta\left(  L^{p}%
(\Omega,X)\right)  \text{ satisfy suitable conditions}\},
\]
where $\beta\left(  L^{p}(\Omega,X)\right)  $ is the set of all non-empty
bounded subsets of $L^{p}(\Omega,X)$. The collection $\mathcal{D}$ is said to
be inclusion-closed if $D\in\mathcal{D}$ and $B(t)\subset D(t)$, $B(t)\in
\beta\left(  L^{p}(\Omega,X)\right)  $, for all $t\in\mathbb{R}$, imply that
$B=\{B(t)\}_{t\in\mathbb{R}}\in\mathcal{D}.$

A family $D=\{D(t)\}_{t\in\mathbb{R}}$ is said to be compact (weakly compact,
bounded, etc.) if each set $D(t)$ is compact (weakly compact, bounded, etc.).

A family of maps $\Phi:\mathbb{R}^{+}\times\mathbb{R}\times L^{p}%
(\Omega,X)\rightarrow L^{p}(\Omega,X)$ is called a mean random dynamical
system if:

\begin{itemize}
\item $\Phi(0,\tau)$ is the identity map for all $\tau\in\mathbb{R};$

\item $\Phi(t+s,\tau)=\Phi(t,s+\tau)\circ\Phi(s,\tau)$ for all $t,s\in
\mathbb{R}^{+}$, $\tau\in\mathbb{R}$.
\end{itemize}

\begin{definition}
A family $K=\{K(t)\}_{t\in\mathbb{R}}\in\mathcal{D}$ is called a $\mathcal{D}%
$-pullback weakly attracting set if for all $\tau\in\mathbb{R}$,
$D\in\mathcal{D}$ and any weak neighborhood $\mathcal{N}^{w}(K(\tau))$ of
$K(\tau)$ there exists $T(\tau,D,\mathcal{N}^{w}(K(\tau)))>0$ such that%
\[
\Phi(t,\tau-t)D(\tau-t)\subset\mathcal{N}^{w}(K(\tau))
\]
as soon as $t\geq T.$
\end{definition}

We introduce the main concept of weak $\mathcal{D}$-pullback mean random attractor.

\begin{definition}
A family $\mathcal{A}=\{\mathcal{A}(t)\}_{t\in\mathbb{R}}\in\mathcal{D}$ is
called a weak $\mathcal{D}$-pullback mean random attractor if:

\begin{enumerate}
\item $\mathcal{A}$ is weakly compact.

\item $\mathcal{A}$ is $\mathcal{D}$-pullback weakly attracting.

\item $\mathcal{A}$ is minimal, that is, if $B\in\mathcal{D}$ is weakly
compact and $\mathcal{D}$-pullback weakly attracting, then $\mathcal{A}%
(t)\subset B(t)$ for all $t\in\mathbb{R}$.
\end{enumerate}
\end{definition}

It follows from this definition that a weak $\mathcal{D}$-pullback mean random
attractor is unique if it exists.

Further we need the concept of $\mathcal{D}$-pullback absorbing family.

\begin{definition}
A family $K=\{K(t)\}_{t\in\mathbb{R}}\in\mathcal{D}$ is called a $\mathcal{D}%
$-pullback absorbing set if for all $\tau\in\mathbb{R}$ and $D\in\mathcal{D}$
there exists $T(\tau,D)>0$ such that%
\[
\Phi(t,\tau-t)D(\tau-t)\subset K(\tau)
\]
as soon as $t\geq T.$
\end{definition}

\begin{theorem}
\label{AttrExist}\cite[p.2183]{Wang} Let $X$ be reflexive. Assume that
$\mathcal{D}$ is an inclusion-closed collection. If the mean random dynamical
system $\Phi$ possesses a weakly compact $\mathcal{D}$-pullback absorbing
family $K=\{K(t)\}_{t\in\mathbb{R}}\in\mathcal{D}$, then $\Phi$ has a unique
weak $\mathcal{D}$-pullback mean random attractor given by%
\begin{equation}
\mathcal{A}(t)=\cap_{r\geq0}\overline{\cup_{s\geq r}\Phi(s,t-s)K(t-s)}%
^{w}\text{, }t\in\mathbb{R}\text{,} \label{AttrCharac}%
\end{equation}
where $\overline{C}^{w}$ means the closure of $C$ in the weak topology of
$L^{p}(\Omega,X)$.
\end{theorem}

\bigskip

The weak $\mathcal{D}$-pullback mean random attractor is called invariant if
\[
\mathcal{A}(t)=\Phi(\tau,t-\tau,\mathcal{A}(t-\tau))\text{ for all }\tau
\geq0,\ t\in\mathbb{R}\text{. }%
\]
The mean random dynamical system $\Phi$ is weakly continuous if the map
$\Phi(t,\tau):L^{p}(\Omega,X)\rightarrow L^{p}(\Omega,X)$ is weakly continuous
for any $t\geq0$, $\tau\in\mathbb{R}$.

\begin{lemma}
Let the conditions of Theorem \ref{AttrExist} hold true. Assume that $\Phi$ is
weakly continuous. Then the weak $\mathcal{D}$-pullback mean attractor
$\mathcal{A}$ is invariant.
\end{lemma}

\begin{proof}
Let $y\in\mathcal{A}(t)$. In view of characterization (\ref{AttrCharac}) and
arguing as in \cite[Lemma 3.3]{KMV03} there exist nets $s_{\alpha}$,
$y_{\alpha}\in\Phi(s_{\alpha},t-s_{\alpha})K(t-s_{\alpha})$ such that
$s_{\alpha}\rightarrow+\infty$ and $y_{\alpha}\rightarrow y$ weakly in
$L^{p}(\Omega,X)$. Hence, $y_{\alpha}=\Phi(s_{\alpha},t-s_{\alpha})z_{\alpha}$
for some net $z_{\alpha}\in K(t-s_{\alpha})$. Thus,
\[
y_{\alpha}=\Phi(s_{\alpha},t-s_{\alpha})z_{\alpha}=\Phi(\tau,t-\tau
)\Phi(s_{\alpha}-\tau,t-s_{\alpha})z_{\alpha}=\Phi(\tau,t-\tau)x_{\alpha}.
\]
There exists $s_{0}$ such that $x_{\alpha}\in K(t-\tau)$ for $s_{\alpha}\geq
s_{0}$. Since $K(t-\tau)$ is bounded, passing to a subnet we have that
$x_{\alpha}\rightarrow x$ for some $x$. The weak continuity of $\Phi$ implies
then that $y=\Phi(\tau,t-\tau)x$. Note that
\[
x_{\alpha}\in\Phi(s_{\alpha}-\tau,t-s_{\alpha})K(t-s_{\alpha})=\Phi
(\widetilde{s}_{\alpha},t-\tau-\widetilde{s}_{\alpha})K(t-\tau-\widetilde{s}%
_{\alpha}),
\]
so
\[
x\in\cap_{r\geq0}\overline{\cup_{s\geq r}\Phi(s,t-\tau-s)K(t-\tau-s)}%
^{w}=\mathcal{A}(t-\tau).
\]
We have obtained that $\mathcal{A}(t)\subset\Phi(\tau,t-\tau)\mathcal{A}%
(t-\tau).$

Conversely, for any weak neighborhood $\mathcal{N}^{w}(\mathcal{A}(t))$ we
have%
\begin{align*}
\Phi(\tau,t-\tau)\mathcal{A}(t-\tau)  &  \subset\Phi(\tau,t-\tau)\Phi
(r,t-\tau-r)\mathcal{A}(t-\tau-r)\\
&  =\Phi(\tau+r,t-\tau-r)\mathcal{A}(t-\tau-r)\subset\mathcal{N}%
^{w}(\mathcal{A}(t)),
\end{align*}
for $r$ large enough, so $\Phi(\tau,t-\tau)\mathcal{A}(t-\tau)\subset
\mathcal{A}(t)$.
\end{proof}

\bigskip

The map $\phi:\mathbb{R}\rightarrow L^{p}(\Omega,X)$ is a complete trajectory
if $\phi(t)=\Phi(t-s,s)\phi(s)$ for all $s<t$. We can characterize the
attractor in terms of complete trajectories belonging to $\mathcal{D}$.

\begin{lemma}
Assume that $\mathcal{D}$ is an inclusion-closed collection. Assume that
$\Phi$ possesses an invariant weak $\mathcal{D}$-pullback mean attractor
$\mathcal{A}\in\mathcal{D}$. Then%
\begin{equation}
\mathcal{A}(t)=\{\phi(t):\phi\in\mathcal{D}\text{ is a complete trajectory}\}.
\label{Charac}%
\end{equation}

\end{lemma}

\begin{proof}
If $\phi\in\mathcal{D}$ is a complete trajectory, then
\[
\phi(t)=\Phi(s,t-s,\phi(t-s))\text{ for all }s\geq0,\ t\in\mathbb{R}\text{, }%
\]
so for any $t\in\mathbb{R}$ and any weak neighborhood $\mathcal{N}%
^{w}(\mathcal{A}(t))$ there exists $T=T(t,\phi,\mathcal{N}^{w}(\mathcal{A}%
(t)))$ such that%
\[
\phi(t)=\Phi(s,t-s,\phi(t-s))\in\mathcal{N}^{w}(\mathcal{A}(t))\text{ for all
}s\geq T,
\]
which implies that $\phi(t)\in\mathcal{A}(t).$

Conversely, let $y\in\mathcal{A}(t)$, $t\in\mathbb{R}$ be arbitrary. Since
$\mathcal{A}$ is invariant, we have%
\[
y\in\mathcal{A}(t)=\Phi(1,t-1,\mathcal{A}(t-1)),
\]
so there exists $z_{1}\in\mathcal{A}(t-1)$ such that $y=\Phi(1,t-1,z_{1})$. We
put $\phi^{1}(r)=\Phi(r-t+1,t-1,z_{1})$ for all $r\geq t-1$. By the invariance
of the attractor, $\phi^{1}(r)\in\mathcal{A}(r)$ for all $r\geq t-1$. Also,
$\phi^{1}(r)=\Phi(r-s,s,\phi^{1}(s))$ for any $r\geq s\geq t-1$. By the same
argument, there is $z_{2}\in\mathcal{A}(t-2)$ such that $z_{1}=\Phi
(1,t-2,z_{2})$. The function $\phi^{2}(r)=\Phi(r-t+2,t-2,z_{2}),$ for all
$r\geq t-2,$ satisfies $\phi^{2}(r)\in\mathcal{A}(r),$ for all $r\geq t-2,$
$\phi^{2}(r)=\Phi(r-s,s,\phi^{2}(s)),$ for any $r\geq s\geq t-2,$ and
$\phi^{2}(r)=\phi^{1}(r)$ for all $r\geq t-1$. In this way, we construct a
sequence of functions $\phi^{k}:[t-k,+\infty)\rightarrow L^{p}(\Omega,X)$ such
that $\phi^{k}(r)\in\mathcal{A}(r),$ for all $r\geq t-k,$ $\phi^{k}%
(r)=\Phi(r-s,s,\phi^{k}(s)),$ for any $r\geq s\geq t-k,$ and $\phi^{k}%
(r)=\phi^{k-1}(r)$ for all $r\geq t-k+1$. Let $\phi:\mathbb{R}\rightarrow
L^{p}(\Omega,X)$ be the common value of the functions $\phi^{k}$ at any point
$t\in\mathbb{R}$. It is clear that $\phi$ is a complete trajectory satisfying
$\phi(t)\in\mathcal{A}(r)$ for all $t$. Also, $\phi\in\mathcal{D}$ because
$\mathcal{A}\in\mathcal{D}$ and $\mathcal{D}$ is inclusion closed.
\end{proof}

\bigskip

A\ family of sets $D$ is said to be backwards bounded if there is $t_{0}%
\in\mathbb{R}$ such that $\cup_{t\leq t_{0}}D(t)$ is bounded. It is called
bounded if $\cup_{t\in\mathbb{R}}D(t)$ is bounded.

\begin{lemma}
Assume that $\mathcal{D}$ is an inclusion-closed collection and that $\Phi$
possesses an invariant weak $\mathcal{D}$-pullback mean attractor
$\mathcal{A}\in\mathcal{D}$ which is backwards bounded. Then%
\begin{equation}
\mathcal{A}(t)=\{\phi(t):\phi\in\mathcal{D}\text{ is a backwards bounded
complete trajectory}\}. \label{charac2}%
\end{equation}
If, moreover, $\mathcal{D}$ contains any backwards bounded family, then%
\begin{equation}
\mathcal{A}(t)=\{\phi(t):\phi\text{ is a backwards bounded complete
trajectory}\}. \label{charac3}%
\end{equation}

\end{lemma}

\begin{proof}
The first statement follows directly from (\ref{Charac}) and the fact that
$\mathcal{A}$ is backwards bounded.

The second one follows because any backwards bounded complete trajectory has
to belong to $\mathcal{D}$, so the two sets defined in (\ref{charac2}) and
(\ref{charac3}) coincide.
\end{proof}

\bigskip

In the same way we prove the following.

\begin{lemma}
Assume that $\Phi$ possesses an invariant weak $\mathcal{D}$-pullback mean
attractor $\mathcal{A}\in\mathcal{D}$ which is bounded. Then%
\begin{equation}
\mathcal{A}(t)=\{\phi(t):\phi\in\mathcal{D}\text{ is a bounded complete
trajectory}\}. \label{charac4}%
\end{equation}
If, moreover, $\mathcal{D}$ contains any bounded family, then%
\begin{equation}
\mathcal{A}(t)=\{\phi(t):\phi\text{ is a bounded complete trajectory}\}.
\label{charac5}%
\end{equation}

\end{lemma}

\subsection{Weak pullback mean attractors over filtered probability spaces}

In this subsection we recall the main results of the theory of weak
$\mathcal{D}$-pullback mean random attractors for mean random dynamical
systems over filtered probability spaces.

As usual, we denote by $(\Omega,\mathcal{F},\{\mathcal{F}_{t}\}_{t\in
\mathbb{R}},P)$ a complete filtered probability space, where $\{\mathcal{F}%
_{t}\}_{t\in\mathbb{R}}$ is an increasing right continuous family of
sub-$\sigma$-algebras of $\mathcal{F}$ which contains all $P$-null sets. For
$p\in(1,\infty)$ let $L^{p}(\Omega,\mathcal{F}_{t};X)$ be the subspace of
(class of) functions $f$ of $L^{p}(\Omega,\mathcal{F};X)$ such that $f$ is
strongly $\mathcal{F}_{t}$-measurable.

As before, we denote by $\mathcal{D}$ a collection of some families
$D=\{D(t)\}_{t\in\mathbb{R}}$ of non-empty bounded sets $D(t)\subset
L^{p}(\Omega,\mathcal{F}_{t};X)$:%
\[
\mathcal{D}=\{D=\{D(t)\}_{t\in\mathbb{R}}:D(t)\in\beta\left(  L^{p}%
(\Omega,\mathcal{F}_{t};X)\right)  \text{ satisfying suitable conditions}\}.
\]

A family $D=\{D(t)\}_{t\in\mathbb{R}}$ is said to be compact (weakly compact,
bounded, etc.) if each set $D(t)$ is compact (weakly compact, bounded, etc.)
in $L^{p}(\Omega,\mathcal{F}_{t};X)$.

\begin{definition}
\label{MRDS2}The family of maps $\Phi(t,\tau):L^{p}(\Omega,\mathcal{F}_{\tau
};X)\rightarrow L^{p}(\Omega,\mathcal{F}_{t+\tau};X)$, $t\in\mathbb{R}^{+}$,
$\tau\in\mathbb{R}$ is called a mean random dynamical system on $L^{p}%
(\Omega,\mathcal{F};X)$ over $(\Omega,\mathcal{F},\{\mathcal{F}_{t}%
\}_{t\in\mathbb{R}},P)$ if:

\begin{itemize}
\item $\Phi(0,\tau)$ is the identity map for all $\tau\in\mathbb{R};$

\item $\Phi(t+s,\tau)=\Phi(t.s+\tau)\circ\Phi(s,\tau)$ for all $t,s\in
\mathbb{R}^{+}$, $\tau\in\mathbb{R}$.
\end{itemize}
\end{definition}

\begin{definition}
The family $K=\{K(t)\}_{t\in\mathbb{R}}\in\mathcal{D}$ is called $\mathcal{D}%
$-pullback weakly attracting for $\Phi$ if for all $\tau\in\mathbb{R}$,
$D\in\mathcal{D}$ and any weak neighborhood $\mathcal{N}^{w}(K(\tau))$ of
$K(\tau)$ in $L^{p}(\Omega,\mathcal{F}_{\tau};X)$ there exists $T(\tau
,D,\mathcal{N}^{w}(K(\tau)))>0$ such that%
\[
\Phi(t,\tau-t)D(\tau-t)\subset\mathcal{N}^{w}(K(\tau))
\]
as soon as $t\geq T.$
\end{definition}

We introduce the main concept of weak $\mathcal{D}$-pullback mean random attractor.

\begin{definition}
A family $\mathcal{A}=\{\mathcal{A}(t)\}_{t\in\mathbb{R}}\in\mathcal{D}$ is
called a weak $\mathcal{D}$-pullback mean random attractor for $\Phi$ if:

\begin{enumerate}
\item $\mathcal{A}(\tau)$ is weakly compact in $L^{p}(\Omega,\mathcal{F}%
_{\tau};X)$ for all $\tau\in\mathbb{R}$.

\item $\mathcal{A}$ is $\mathcal{D}$-pullback weakly attracting.

\item $\mathcal{A}$ is minimal, that is, if $B\in\mathcal{D}$ satisfies
conditions 1-2, then $\mathcal{A}(t)\subset B(t)$ for all $t\in\mathbb{R}$.
\end{enumerate}
\end{definition}

It follows from this definition that a weak $\mathcal{D}$-pullback mean random
attractor is unique if it exists.

\begin{definition}
The family $K=\{K(t)\}_{t\in\mathbb{R}}\in\mathcal{D}$ is called $\mathcal{D}%
$-pullback absorbing for $\Phi$ if for all $\tau\in\mathbb{R}$ and
$D\in\mathcal{D}$ there exists $T(\tau,D)>0$ such that%
\[
\Phi(t,\tau-t)D(\tau-t)\subset K(\tau)
\]
as soon as $t\geq T.$
\end{definition}

\begin{theorem}
\label{AttrExist2}\cite[p.2188]{Wang} Let $X$ be reflexive. Assume that
$\mathcal{D}$ is an inclusion-closed collection. If the mean random dynamical
system $\Phi$ possesses a weakly compact $\mathcal{D}$-pullback absorbing
family $K=\{K(t)\}_{t\in\mathbb{R}}\in\mathcal{D}$, then $\Phi$ has a unique
weak $\mathcal{D}$-pullback mean random attractor given by%
\begin{equation}
\mathcal{A}(t)=\cap_{r\geq0}\overline{\cup_{s\geq r}\Phi(s,t-s)K(t-s)}%
^{w}\text{, }t\in\mathbb{R}\text{,} \label{AttrCharac2}%
\end{equation}
where the closure is taken with respect to the weak topology of $L^{p}%
(\Omega,\mathcal{F}_{t};X)$.
\end{theorem}

As before, the weak $\mathcal{D}$-pullback mean random attractor is called
invariant if
\[
\mathcal{A}(t)=\Phi(\tau,t-\tau,\mathcal{A}(t-\tau))\text{ for all }\tau
\geq0,\ t\in\mathbb{R}\text{. }%
\]
The mean random dynamical system on $L^{p}(\Omega,\mathcal{F};X)$ over
$(\Omega,\mathcal{F},\{\mathcal{F}_{t}\}_{t\in\mathbb{R}},P)$ $\Phi$ is weakly
continuous if the map $\Phi(t,\tau):L^{p}(\Omega,\mathcal{F}_{\tau
};X)\rightarrow L^{p}(\Omega,\mathcal{F}_{t+\tau};X)$ is weakly continuous for
any $t\geq0$, $\tau\in\mathbb{R}$.

\begin{lemma}
Let the conditions of Theorem \ref{AttrExist2} hold true. Assume that $\Phi$
is weakly continuous. Then the weak $\mathcal{D}$-pullback mean attractor
$\mathcal{A}$ is invariant.
\end{lemma}

\begin{proof}
Let $y\in\mathcal{A}(t)$. In view of characterization (\ref{AttrCharac2}) and
arguing as in \cite[Lemma 3.3]{KMV03} there exist nets $s_{\alpha}$,
$y_{\alpha}\in\Phi(s_{\alpha},t-s_{\alpha})K(t-s_{\alpha})$ such that
$s_{\alpha}\rightarrow+\infty$ and $y_{\alpha}\rightarrow y$ weakly in
$L^{p}(\Omega,\mathcal{F}_{t};X)$. Hence, $y_{\alpha}=\Phi(s_{\alpha
},t-s_{\alpha})z_{\alpha}$ for some net $z_{\alpha}\in K(t-s_{\alpha})$.
Thus,
\[
y_{\alpha}=\Phi(s_{\alpha},t-s_{\alpha})z_{\alpha}=\Phi(\tau,t-\tau
)\Phi(s_{\alpha}-\tau,t-s_{\alpha})z_{\alpha}=\Phi(\tau,t-\tau)x_{\alpha}.
\]
There exists $s_{0}$ such that $x_{\alpha}\in K(t-\tau)$ for $s_{\alpha}\geq
s_{0}$. Since $K(t-\tau)$ is bounded, passing to a subnet we have that
$x_{\alpha}\rightarrow x$ for some $x$. The weak continuity of $\Phi$ implies
then that $y=\Phi(\tau,t-\tau)x$. Note that
\[
x_{\alpha}\in\Phi(s_{\alpha}-\tau,t-s_{\alpha})K(t-s_{\alpha})=\Phi
(\widetilde{s}_{\alpha},t-\tau-\widetilde{s}_{\alpha})K(t-\tau-\widetilde{s}%
_{\alpha}),
\]
so
\[
x\in\cap_{r\geq0}\overline{\cup_{s\geq r}\Phi(s,t-\tau-s)K(t-\tau-s)}%
^{w}=\mathcal{A}(t-\tau).
\]
We have obtained that $\mathcal{A}(t)\subset\Phi(\tau,t-\tau)\mathcal{A}%
(t-\tau).$

Conversely, for any weak neighborhood $\mathcal{N}^{w}(\mathcal{A}(t))$ we
have%
\begin{align*}
\Phi(\tau,t-\tau)\mathcal{A}(t-\tau)  &  \subset\Phi(\tau,t-\tau)\Phi
(r,t-\tau-r)\mathcal{A}(t-\tau-r)\\
&  =\Phi(\tau+r,t-\tau-r)\mathcal{A}(t-\tau-r)\subset\mathcal{N}%
^{w}(\mathcal{A}(t)),
\end{align*}

\end{proof}

\bigskip

A map $\phi$ such that $\phi(t)\in L^{p}(\Omega,\mathcal{F}_{t};X)$ for all
$t\in\mathbb{R}$ is a complete trajectory if $\phi(t)=\Phi(t-s,s)\phi(s)$ for
all $s<t$. We can characterize the attractor in terms of complete trajectories
belonging to $\mathcal{D}$.

\begin{lemma}
Assume that $\mathcal{D}$ is an inclusion-closed collection. Assume that
$\Phi$ possesses an invariant weak $\mathcal{D}$-pullback mean attractor
$\mathcal{A}\in\mathcal{D}$. Then%
\begin{equation}
\mathcal{A}(t)=\{\phi(t):\phi\in\mathcal{D}\text{ is a complete trajectory}\}.
\label{Charac2}%
\end{equation}

\end{lemma}

\begin{proof}
If $\phi\in\mathcal{D}$ is a complete trajectory, then
\[
\phi(t)=\Phi(s,t-s,\phi(t-s))\text{ for all }s\geq0,\ t\in\mathbb{R}\text{, }%
\]
so for any $t\in\mathbb{R}$ and any weak neighborhood $\mathcal{N}%
^{w}(\mathcal{A}(t))$ there exists $T=T(t,\phi,\mathcal{N}^{w}(\mathcal{A}%
(t)))$ such that%
\[
\phi(t)=\Phi(s,t-s,\phi(t-s))\in\mathcal{N}^{w}(\mathcal{A}(t))\text{ for all
}s\geq T,
\]
which implies that $\phi(t)\in\mathcal{A}(t).$

Conversely, let $y\in\mathcal{A}(t)$, $t\in\mathbb{R}$ be arbitrary. Since
$\mathcal{A}$ is invariant, we have%
\[
y\in\mathcal{A}(t)=\Phi(1,t-1,\mathcal{A}(t-1)),
\]
so there exists $z_{1}\in\mathcal{A}(t-1)$ such that $y=\Phi(1,t-1,z_{1})$. We
put $\phi^{1}(r)=\Phi(r-t+1,t-1,z_{1})$ for all $r\geq t-1$. By the invariance
of the attractor, $\phi^{1}(r)\in\mathcal{A}(r)$ for all $r\geq t-1$. Also,
$\phi^{1}(r)=\Phi(r-s,s,\phi^{1}(s))$ for any $r\geq s\geq t-1$. By the same
argument, there is $z_{2}\in\mathcal{A}(t-2)$ such that $z_{1}=\Phi
(t-1,t-2,z_{2})$. The function $\phi^{2}(r)=\Phi(r-t+2,t-2,z_{1}),$ for all
$r\geq t-2,$ satisfies $\phi^{2}(r)\in\mathcal{A}(r),$ for all $r\geq t-2,$
$\phi^{2}(r)=\Phi(r-s,s,\phi^{2}(s)),$ for any $r\geq s\geq t-2,$ and
$\phi^{2}(r)=\phi^{1}(r)$ for all $r\geq t-1$. In this way, we construct a
sequence of functions $\phi^{k}$ such that $\phi^{k}(r)\in L^{p}%
(\Omega,\mathcal{F}_{r};X)$, $\phi^{k}(r)\in\mathcal{A}(r),$ for all $r\geq
t-k,$ $\phi^{k}(r)=\Phi(r-s,s,\phi^{k}(s)),$ for any $r\geq s\geq t-k,$ and
$\phi^{k}(r)=\phi^{k-1}(r)$ for all $r\geq t-k+1$. Let $\phi$ be the common
value of the functions $\phi^{k}$ at any point $t\in\mathbb{R}$. It is clear
that $\phi$ is a complete trajectory satisfying $\phi(t)\in\mathcal{A}(r)$ for
all $t$. Also, $\phi\in\mathcal{D}$ because $\mathcal{A}\in\mathcal{D}$ and
$\mathcal{D}$ is inclusion closed.
\end{proof}

\bigskip

\begin{lemma}
Assume that $\mathcal{D}$ is an inclusion-closed collection and that $\Phi$
possesses an invariant weak $\mathcal{D}$-pullback mean attractor
$\mathcal{A}\in\mathcal{D}$ which is backwards bounded. Then%
\begin{equation}
\mathcal{A}(t)=\{\phi(t):\phi\in\mathcal{D}\text{ is a backwards bounded
complete trajectory}\}. \label{charac2b}%
\end{equation}
If, moreover, $\mathcal{D}$ contains any backwards bounded family, then%
\begin{equation}
\mathcal{A}(t)=\{\phi(t):\phi\text{ is a backwards bounded complete
trajectory}\}. \label{charac3b}%
\end{equation}

\end{lemma}

\begin{proof}
The first statement follows directly from (\ref{charac2}) and the fact that
$\mathcal{A}$ is backwards bounded.

The second one follows because any backwards bounded complete trajectory has
to belong to $\mathcal{D}$, so the two sets defined in (\ref{charac2b}) and
(\ref{charac3b}) coincide.
\end{proof}

\bigskip

In the same way we prove the following.

\begin{lemma}
Assume that $\Phi$ possesses an invariant weak $\mathcal{D}$-pullback mean
attractor $\mathcal{A}\in\mathcal{D}$ which is bounded. Then%
\[
\mathcal{A}(t)=\{\phi(t):\phi\in\mathcal{D}\text{ is a bounded complete
trajectory}\}.
\]
If, moreover, $\mathcal{D}$ contains any bounded family, then%
\[
\mathcal{A}(t)=\{\phi(t):\phi\text{ is a bounded complete trajectory}\}.
\]

\end{lemma}

\section{The mean random attractor for a non-local problem with random initial
data\label{AttractorRandom}}

We put $H=L^{2}\left(  \mathcal{O}\right)  $ with norm $\left\Vert
\text{\textperiodcentered}\right\Vert $ (we will use $\left\Vert
\text{\textperiodcentered}\right\Vert $ also for the norm in $\left(
L^{2}\left(  \mathcal{O}\right)  \right)  ^{d}$, $d\geq1$) and $V=H_{0}%
^{1}\left(  \mathcal{O}\right)  $ with norm $\left\Vert u\right\Vert
_{V}=\left\Vert \nabla u\right\Vert .$ As usual, $\left(
\text{\textperiodcentered},\text{\textperiodcentered}\right)  $ is the scalar
product in $H^{d},\ d\geq1$, and also the duality between $L^{q}(\mathcal{O})$
and $L^{p}(\mathcal{O})$, where $p\geq2$ and $q$ is its conjugate, that is,
$1/p+1/q=1$. The duality between $V$ and its dual space $V^{\ast}$ is denoted
by $\left\langle \text{\textperiodcentered,\textperiodcentered}\right\rangle
_{V^{\ast},V}$.

Let us consider the following reaction-diffusion equation%
\begin{equation}
\left\{
\begin{array}
[c]{l}%
\dfrac{\partial u}{\partial t}-a(\Vert u\Vert_{V}^{2})\Delta
u=f(u)+h(t,x)\ \text{in }(\tau,\infty)\times\mathcal{O},\\
u=0\quad\text{on }(\tau,\infty)\times\partial\mathcal{O},\\
u(\tau,x)=u_{\tau}(x)\quad\text{for }x\in\mathcal{O},
\end{array}
\right.  \label{RD}%
\end{equation}
where $\mathcal{O}$ is a bounded open set of $\mathbb{R}^{n}$ with smooth
boundary $\partial\mathcal{O}$ and the functions $h$, $a\in C(\mathbb{R}^{+}%
)$, $f\in C^{1}(\mathbb{R})$, satisfy the following assumptions:%
\begin{equation}
f(r)r\leq-\alpha\left\vert r\right\vert ^{p}+\beta,\label{Diss}%
\end{equation}%
\begin{equation}
\left\vert f(r)\right\vert \leq\gamma\left\vert r\right\vert ^{p-1}%
+\delta,\label{Growth}%
\end{equation}%
\begin{equation}
f^{\prime}(r)\leq\eta,\label{Deriv}%
\end{equation}%
\begin{equation}
h\in L_{loc}^{2}(\mathbb{R},H),\label{h1}%
\end{equation}%
\begin{equation}
0<m\leq a\left(  s\right)  \leq M,\label{Cond1a}%
\end{equation}%
\begin{equation}
s\mapsto a(s^{2})s\text{ is non-decreasing,}\label{Cond2a}%
\end{equation}
where $p\geq2$, $r\in\mathbb{R}$, $s\geq0$ and $\alpha,\beta,\gamma
,\delta,\eta>0.$

\begin{remark}
\label{f}Without loss of generality we can assume that $f(0)=0$, as defining
$\overline{f}(u)=f(u)-f(0)$, $\overline{h}(t)=h(t)+f(0)$ we obtain the
equivalent equation%
\[
\dfrac{\partial u}{\partial t}-a(\Vert u\Vert_{V}^{2})\Delta u=\overline
{f}(u)+\overline{h}(t,x)\ \text{in }(\tau,\infty)\times\mathcal{O},
\]
for which (\ref{Diss})-(\ref{Cond2a}) hold and, additionally, $\overline
{f}(0)=0.$
\end{remark}

We will study the existence of solutions of problem (\ref{RD}) for random
initial data $u_{\tau}\in L^{2}(\Omega,H)$ in a probability space
$(\Omega,\mathcal{F},P)$.

For the operator $A=-\Delta$, thanks to the assumptions on the domain $\Omega
$, it is well known that $D(A)=H^{2}(\mathcal{O})\cap H_{0}^{1}(\mathcal{O})$
\cite[Proposition 6.19]{robinson}.

\begin{definition}
\label{defsol}Let be given $\tau\in\mathbb{R}$ and $u_{\tau}\in L^{2}%
(\Omega,H)$. A continuous mapping $u:[\tau,\infty)\rightarrow L^{2}(\Omega,H)$
is called a regular solution of problem (\ref{RD}) if
\[
u\in C([\tau,\infty),L^{2}(\Omega,H))\cap L_{loc}^{2}(\tau,\infty;L^{2}%
(\Omega,V))\cap L_{loc}^{p}(\tau,\infty;L^{p}(\Omega,L^{p}(\mathcal{O}))),
\]%
\[
u\in L^{\infty}(\tau+\varepsilon,T;L^{2}(\Omega,V)\cap L^{2}(\tau
+\varepsilon,T;L^{2}(\Omega,D(A)),\text{ }\forall\ 0<\varepsilon<T<\infty,
\]
and $u$ satisfies, $P$-a.s., the equality%
\begin{equation}%
\begin{split}
&  (u(t),\xi)+\int_{\tau}^{t}a(\Vert u(s)\Vert_{H_{0}^{1}}^{2})(\nabla
u(s),\nabla\xi)ds\\
=(u_{0}, &  \xi)+\int_{\tau}^{t}\int_{\mathcal{O}}f(u(s,x))\ \xi
(x)dxds+\int_{\tau}^{t}\int_{\mathcal{O}}h(s,x)\xi(x)dxds,
\end{split}
\label{rmequa}%
\end{equation}
for every $t>\tau$ and $\xi\in V\cap L^{p}(\mathcal{O})$.
\end{definition}

\begin{theorem}
Suppose that (\ref{Diss})-(\ref{Cond2a}) hold true. Then for every $\tau
\in\mathbb{R}$ and $u_{\tau}\in L^{2}(\Omega,H)$, problem (\ref{RD}) has a
unique regular solution $u(\cdot)$, which is continuous with respect to the
initial datum $u_{\tau}$ in $L^{2}(\Omega,H)$. Moreover, it satisfies the
energy equality%
\begin{equation}
\frac{d}{dt}\mathbb{E}(\left\Vert u(t)\right\Vert ^{2})+2\mathbb{E}%
(a(\left\Vert u(t)\right\Vert _{V}^{2})\left\Vert u(t)\right\Vert _{V}%
^{2})=2\mathbb{E}(\left(  f(u(t),u(t)\right)  +(h(t),u(t))), \label{Energy}%
\end{equation}
for a.a. $t>\tau.$
\end{theorem}

\begin{proof}
We will prove the result using the Faedo-Galerkin method.

Consider a fixed value $T>0.$ Let $\{e_{j}\}_{j\geq1}$ be a sequence of
eigenfunctions of $-\Delta$ in $V$ with homogeneous Dirichlet boundary
conditions, which forms a special basis of $L^{2}(\mathcal{O})$.

We need to ensure that the eigenfunctions are elements of $L^{p}(\mathcal{O}%
)$. Indeed, by the Sobolev embedding theorem, we have
\[
H^{s}(\mathcal{O})\subset L^{p}(\mathcal{O})\quad\text{ for }s\geq n(p-2)/2p.
\]
Taking $A=-\Delta$, we define the domain of a fractional power of $A$ as
\[
D(A^{s/2})=\{u\in L^{2}(\mathcal{O}):\sum_{j=1}^{\infty}\lambda_{j}%
^{s}(u,e_{j})^{2}<\infty\},
\]
where $\lambda_{j}$ is the eigenvalue associated to $e_{j}.$ Also,
$\{e_{j}\}\in D(A^{s/2})$. If we assume $\mathcal{O}$ to be a bounded $C^{s}$
domain (smoothness condition on the domain), by Theorem 6.18 in
\cite{robinson} we have that $D(A^{s/2})\subset H^{s}(\mathcal{O})$ and so
$\{e_{j}\}\in L^{p}(\mathcal{O}).$

Therefore, we can consider $\{e_{j}\}\subset V\cap L^{p}(\mathcal{O})$ a basis
of $L^{2}(\mathcal{O})$, with $s\geq\max\{n(p-2)/2p,1\}$. By this way,
$H_{0}^{s}(\mathcal{O})\subset V\cap L^{p}(\mathcal{O})$ and the set
$\cup_{n\in\mathbb{N}}V_{n}$ is dense in $L^{2}(\mathcal{O})$ and also in
$V\cap L^{p}(\mathcal{O})$ \cite{LSU}, where $V_{n}=span[e_{1},\ldots,e_{n}]$.

As usual, $P_{n}$ is the orthogonal projection in $H$, that is
\[
u_{n}:=P_{n}u=\sum_{j=1}^{n}(u,e_{j})e_{j},\quad\forall u\in H.
\]

Let $u_{\tau}:\Omega\mapsto H$ be a $\mathcal{F}$-measurable mapping such that
$\mathbb{E}(\Vert u_{\tau}\Vert^{2}))<\infty$. Then for every fixed $\omega
\in\Omega$ and for each integer $n\geq1$, we consider the Galerkin
approximations
\[
u_{n}(t,\omega)=\sum_{j=1}^{n}\gamma_{nj}(t,\omega)e_{j},
\]
which satisfy the following deterministic system parametrized by $\omega$:
\begin{equation}
\left\{
\begin{array}
[c]{l}%
\dfrac{d}{dt}(u_{n},e_{i})+a(\Vert u_{n}\Vert_{V}^{2})(\nabla u_{n},\nabla
e_{i})=(f(u_{n}),e_{i})+(h,e_{i}),\ \forall i=1,\ldots,n,\\
u_{n}(\tau,\omega)=P_{n}u_{\tau}(\omega).
\end{array}
\right.  \label{2121}%
\end{equation}

Using the fact that the eigenfunctions $\{e_{j}\}$ are orthonormal, we obtain
that (\ref{2121}) is equivalent to the Cauchy problem
\begin{equation}%
\begin{split}
\frac{d\gamma_{n_{j}}}{dt} &  =-a(\Vert u_{n}\Vert_{V}^{2})\lambda_{j}%
\gamma_{n_{j}}+(f(u_{n}),e_{j})+(h,e_{j}),\\
(u_{n}(\tau,\omega),e_{j}) &  =(P_{n}u_{\tau}(\omega),e_{j}),\quad
j=1,\ldots,n.
\end{split}
\label{1.7}%
\end{equation}

Since the right hand side of (\ref{1.7}) is continuous in $u_{n}(t)$, for
every fixed $\omega\in\Omega$ and $\tau\in\mathbb{R}$ this Cauchy problem
possesses a solution on some interval $[\tau,t_{n}),\tau<t_{n}<T$ \cite[cf. p.
51]{robinson}. In addition, for each $t\geq\tau$, $u_{n}(t,\omega)$ is
$\mathcal{F}-$measurable with respect to $\omega\in\Omega$. Indeed, since
$u_{n}(t,\omega)$ can be written as $u_{n}(t,\tau,u_{\tau}(\omega))$, the
result follows since $u_{n}$ is continuous and $u_{\tau}$ is measurable
\cite[Lemma 8.2.3]{Frankowska}.

We claim that this solution can be extended to any $[0,T]$ with $T>0$. This
will follow from a priori estimates in the space $H$ of the sequence
$\{u_{n}\}$.

Multiplying by $\gamma_{ni}(t,\tau,\omega)$ and summing from $i=1$ to $n$, we
obtain
\begin{equation}%
\begin{split}
&  \frac{1}{2}\frac{d}{dt}\Vert u_{n}(t,\omega)\Vert^{2}+a(\Vert u_{n}%
\Vert_{V}^{2})\Vert u_{n}(t,\omega)\Vert_{V}^{2}\\
&  =(f(u_{n}(t,\omega)),u_{n}(t,\omega))+(h(t),u_{n}(t,\omega))\
\end{split}
\label{1.8}%
\end{equation}
for $\mathit{a.e.}$ $t\in(0,t_{n})$. \newline Using (\ref{Diss}) and the Young
and Poincar\'{e} inequalities we deduce that
\[
(f(u_{n}(t,\omega)),u_{n}(t,\omega))\leq\beta|\mathcal{O}|-\alpha\Vert
u_{n}(t,\omega)\Vert_{L^{p}(\mathcal{O})}^{p},
\]%
\[
(h(t),u_{n}(t,\omega))\leq\frac{m}{2}\Vert u_{n}(t,\omega)\Vert_{V}^{2}%
+\frac{1}{2\lambda_{1}m}\Vert h(t)\Vert^{2}.
\]
\newline Hence, from (\ref{1.8}) it follows that%
\begin{equation}
\frac{1}{2}\frac{d}{dt}\Vert u_{n}(t,\omega)\Vert^{2}+\frac{m}{2}\Vert
u_{n}(t,\omega)\Vert_{V}^{2}+\alpha\Vert u_{n}(t,\omega)\Vert_{L^{p}%
(\mathcal{O})}^{p}\leq\beta|\mathcal{O}|+\frac{1}{2\lambda_{1}m}\Vert
h(t)\Vert^{2},\label{1.9}%
\end{equation}
for \textit{a.e.} $t\in(0,t_{n})$.

Then, integrating (\ref{1.9}) from $\tau$ to $t\in(\tau,t_{n})$ we deduce
\begin{equation}%
\begin{split}
&  \frac{1}{2}\Vert u_{n}(t,\omega)\Vert^{2}+\frac{m}{2}\int_{\tau}^{t}\Vert
u_{n}(s,\omega)\Vert_{V}^{2}ds+\alpha\int_{\tau}^{t}\Vert u_{n}(s,\omega
)\Vert_{L^{p}(\mathcal{O})}^{p}ds\\
&  \leq\beta|\mathcal{O}|(t-\tau)+\frac{1}{2\lambda_{1}m}\int_{\tau}^{t}\Vert
h(s)\Vert^{2}ds+\frac{1}{2}\Vert u_{n}(\tau,\omega)\Vert^{2}\\
&  \leq TK_{1}+K_{2}(T)+\frac{1}{2}\Vert u_{n}(\tau,\omega)\Vert^{2}.
\end{split}
\label{1.10.1}%
\end{equation}
Since $P_{n}u_{\tau}(\omega)\rightarrow u_{\tau}(\omega)$ in $H$, for every
fixed $\tau\in\mathbb{R},\ \omega\in\Omega$ and $T>0$, the sequence
$\{u_{n}(\cdot,\omega)\}$ is well defined and bounded in $L^{\infty}(\tau
,\tau+T;H))\cap L^{2}(\tau,\tau+T;V)\cap L^{p}(\tau,\tau+T;L^{p}(\mathcal{O}%
)$. Also, $\{-\Delta u_{n}\}$ is bounded in $L^{2}(\tau,\tau+T;V^{\ast}).$

On the other hand, by (\ref{Growth}) it follows that%
\[
\int_{\tau}^{\tau+T}\int_{\mathcal{O}}|f(u_{n}(s,\omega))|^{q}dxds\leq
2^{q-1}C^{q}\left(  |\mathcal{O}|T+\int_{\tau}^{\tau+T}\Vert u_{n}%
(s,\omega)\Vert_{L^{p}(\mathcal{O})}^{p}ds\right)  ,
\]
with $\frac{1}{p}+\frac{1}{q}=1.$ Hence, since $\{u_{n}\}$ is bounded in
$L^{p}(\tau,\tau+T;L^{p}(\mathcal{O}))$, $\{f(u_{n})\}$ is bounded in
$L^{q}(\tau,\tau+T;L^{q}(\mathcal{O}))$.

Now, multiplying (\ref{2121}) by $\lambda_{i}\gamma_{ni}(t)$, summing from
$i=1$ to $n$ and using the Young inequality we obtain%
\begin{equation}%
\begin{split}
&  \frac{1}{2}\frac{d}{dt}\Vert u_{n}(t,\omega)\Vert_{V}^{2}+m\Vert\Delta
u_{n}(t,\omega)\Vert^{2}\\
&  \leq(f(u_{n}(t,\omega)),-\Delta u_{n}(t,\omega))+(h(t),-\Delta
u_{n}(t,\omega))\\
&  \leq\eta\Vert u_{n}(t,\omega)\Vert_{V}^{2}+\frac{1}{2m}\Vert h(t)\Vert
^{2}+\frac{m}{2}\Vert\Delta u_{n}(t,\omega)\Vert^{2},
\end{split}
\label{Inequn}%
\end{equation}
where we have supposed by Remark \ref{f} that $f(0)=0$. Hence, if we apply the
Uniform Gronwall Lemma to the following inequality
\[
\frac{d}{dt}\Vert u_{n}(t,\omega)\Vert_{V}^{2}\leq2\eta\Vert u_{n}%
(t,\omega)\Vert_{V}^{2}+\frac{1}{m}\Vert h(t)\Vert^{2},
\]
in view of (\ref{1.10.1}) for $r>0$ we obtain that
\begin{equation}
\Vert u_{n}(t,\omega)\Vert_{V}^{2}\leq\left(  \frac{2TK_{1}+2K_{2}(T)+\Vert
u_{n}(\tau,\omega)\Vert^{2}}{mr}+K_{3}(T)\right)  e^{2\eta r}\label{Inequn2}%
\end{equation}
for $t\geq\tau+r=t_{1}$. Therefore,
\[
\{\Vert u_{n}(\cdot,\omega)\Vert_{V}\}\text{ is uniformly bounded in }%
[t_{1},\tau+T]
\]
and by the continuity of the function $a$ we get that
\[
\{a(\Vert u_{n}(\cdot,\omega)\Vert_{V}^{2})\}\text{ is bounded in }[t_{1}%
,\tau+T].
\]
Also, it follows that
\begin{equation}
\{u_{n}(\cdot,\omega)\}\text{ is bounded in }L^{\infty}(t_{1},\tau
+T;V).\label{coninfi}%
\end{equation}
On the other hand, by (\ref{Inequn2}) and integrating in (\ref{Inequn}) we
obtain that%
\begin{align}
&  m\int_{t_{1}}^{T}\Vert\Delta u_{n}(t,\omega)\Vert^{2}dt\nonumber\\
&  \leq\frac{1}{2}\Vert u_{n}(t_{1},\omega)\Vert_{V}^{2}+\eta\int_{t_{1}}%
^{T}\Vert u_{n}(t,\omega)\Vert_{V}^{2}dt+\frac{1}{2m}\int_{t_{1}}^{T}\Vert
h(t)\Vert^{2}dt\nonumber\\
&  \leq K_{4}(T,r)(1+\Vert u_{n}(\tau,\omega)\Vert^{2})+\eta\int_{t_{1}}%
^{T}\Vert u_{n}(t,\omega)\Vert_{V}^{2}dt+\frac{1}{2m}\int_{t_{1}}^{T}\Vert
h(t)\Vert^{2}dt,\label{Inequn3}%
\end{align}
so by (\ref{1.10.1})%
\begin{equation}
\{u_{n}(\cdot,\omega)\}\text{ is bounded in }L^{2}(t_{1},\tau
+T;D(A)).\label{cotda}%
\end{equation}
This implies that
\[
\{-\Delta u_{n}(\cdot,\omega)\}
\]
and
\[
\{a(\Vert u_{n}(\cdot,\omega)\Vert_{V}^{2})\Delta u_{n}(\cdot,\omega)\}
\]
are bounded in $L^{2}(t_{1},\tau+T;L^{2}(\mathcal{O}))$.

Thus,
\begin{equation}
\left\{  \dfrac{du_{n}(\cdot,\omega)}{dt}\right\}  \text{ is bounded in }%
L^{q}(t_{1},\tau+T;L^{q}(\mathcal{O})).\label{CotaDeriv}%
\end{equation}
Therefore, there exists $u($\textperiodcentered$,\omega)\in L^{\infty}%
(t_{1},\tau+T;V)\cap L^{2}(\tau,\tau+T;V)\cap L^{\infty}(\tau,\tau+T;H)\cap
L^{2}(t_{1},\tau+T;D(A))\cap L^{p}(\tau,\tau+T;L^{p}(\mathcal{O}))$ such that
$\dfrac{du}{dt}\in L^{q}(t_{1},\tau+T;L^{q}\left(  \mathcal{O}\right)  )$ and
a subsequence $\{u_{n}\}$, relabelled the same, such that (for each $\omega
\in\Omega$)
\begin{equation}%
\begin{split}
u_{n} &  \overset{\ast}{\rightharpoonup}u\text{ in }L^{\infty}(t_{1}%
,\tau+T;V),\\
u_{n} &  \overset{\ast}{\rightharpoonup}u\text{ in }L^{\infty}(\tau
,\tau+T;H),\\
u_{n} &  \rightharpoonup u\text{ in }L^{2}(\tau,\tau+T;V),\\
u_{n} &  \rightharpoonup u\text{ in }L^{p}(\tau,\tau+T;L^{p}(\mathcal{O})),\\
u_{n} &  \rightharpoonup u\text{ in }L^{2}(t_{1},T;D(A)),\\
\frac{du_{n}}{dt} &  \rightharpoonup\frac{du}{dt}\text{ in }L^{q}(t_{1}%
,\tau+T;L^{q}(\mathcal{O})),\\
f(u_{n}) &  \rightharpoonup\chi\text{ in }L^{q}(\tau,\tau+T;L^{q}%
(\mathcal{O})),\\
{a(\Vert u_{n}\Vert_{V}^{2})} &  \overset{\ast}{\rightharpoonup}b\text{ in
}L^{\infty}(t_{1},\tau+T),
\end{split}
\label{1.29}%
\end{equation}
where $\rightharpoonup$ means weak convergence and $\overset{\ast
}{\rightharpoonup}$ weak star convergence. Also, let $t_{0}\in(\tau,\tau+T)$
be fixed. Then, there exists $v\in H$ such that
\begin{equation}
u_{n}(t_{0},\omega)\rightharpoonup v\text{ in }H\label{cont0}%
\end{equation}
for some subsequence. Moreover, by (\ref{cotda})-(\ref{CotaDeriv}) the
Aubin-Lions Compactness Lemma gives that
\[
u_{n}(\cdot,\omega)\rightarrow u(\cdot,\omega)\text{ in }L^{2}(t_{1}%
,\tau+T;V),
\]
so
\[
u_{n}(t,\omega)\rightarrow u(t,\omega)\text{ in }V\text{ a.e. on }(t_{1}%
,\tau+T).
\]
Consequently, by Corollary 1.12 in \cite{robinson}, there exists a subsequence
$\{u_{n}\}$, relabelled the same, such that
\[
u_{n}(t,\omega)(x)\rightarrow u(t,\omega)(x)\text{ a.e. in }(\tau
,\tau+T)\times\mathcal{O}.
\]

Since $f$ is continuous, it follows that
\[
f(u_{n}(t,\omega)(x))\rightarrow f(u(t,\omega)(x))\text{ a.e. in }(\tau
,\tau+T)\times\mathcal{O}.
\]
Therefore, in view of (\ref{1.29}), by \cite[Lemma 1.3]{lions} we have that
$\chi=f(u)$.

As a consequence, by the continuity of $a$, we get that
\[
a(\Vert u_{n}(\cdot,\omega)\Vert_{V}^{2})\rightarrow a(\Vert u(\cdot
,\omega)\Vert_{V}^{2})\quad\text{ a.e. on }(t_{1},\tau+T).
\]
Since the sequence is bounded, by the Lebesgue theorem this convergence takes
place in $L^{2}(t_{1},\tau+T)$ and $b=a(\Vert u\Vert_{H_{0}^{1}}^{2})$ a.e. on
$(t_{1},\tau+T)$.\newline Thus,
\begin{equation}
a(\Vert u_{n}(\cdot,\omega)\Vert_{V}^{2})\Delta u_{n}(\cdot,\omega
)\rightharpoonup a(\Vert u(\cdot,\omega)\Vert_{V}^{2})\Delta u(\cdot
,\omega),\label{convergencianolocal}%
\end{equation}
in $L^{2}(t_{1},\tau+T;H).$ Also, we know \cite[p.224]{robinson} that
\begin{equation}
P_{n}f(u_{n})\rightharpoonup\chi.\label{ConvergPnf}%
\end{equation}

Since $\{e_{i}\}$ is dense in $V\cap L^{p}(\mathcal{O})$, in view of
(\ref{1.29}), (\ref{convergencianolocal}) and (\ref{ConvergPnf}) we can pass
to the limit in (\ref{2121}) and conclude that (\ref{rmequa}) holds for all
$\xi\in V\cap L^{p}(\mathcal{O})$.

We need to guarantee that the initial condition of the problem makes sense. If
$u$ is a weak solution to (\ref{RD}), taking into account (\ref{Growth}) and
(\ref{Cond1a}) it follows that
\begin{equation}
\frac{du}{dt}=a(\Vert u\Vert_{V}^{2})\Delta u+f(u)+h\in L^{2}(\tau
,\tau+T;V^{\ast})+L^{q}(\tau,\tau+T;L^{q}(\mathcal{O})).\label{6b}%
\end{equation}
Therefore, by \cite[p.33]{chepvishik} $u(\cdot,\omega)\in C([\tau,\tau+T],H)$,
so the initial condition makes sense when $u_{\tau}(\omega)\in H.$

We have to check that $u(\tau,\omega)=u_{\tau}(\omega)$ and $u(t_{0}%
,\tau,\omega)=v.$ Indeed, let be $\phi\in C^{1}([\tau,\tau+T];V\cap
L^{p}(\mathcal{O}))$, with $\phi(\tau+T)=0$, $\phi(\tau)\not =0$. Using
(\ref{6b}) we can multiply the equation in (\ref{RD}) by $\phi$ and integrate
by parts in the $t$ variable to obtain that
\begin{equation}%
\begin{split}
&  \int_{\tau}^{\tau+T}\left(  -\left(  u\left(  t,\omega\right)
,\phi^{\prime}\left(  t\right)  \right)  -a(\Vert u(t,\omega\Vert_{V}%
^{2})\left\langle \Delta u\left(  t,\omega\right)  ,\phi\left(  t\right)
\right\rangle _{V^{\ast},V}\right)  dt\\
&  =\int_{\tau}^{\tau+T}\left(  f(u(t,\omega))+h(t),\phi\left(  t\right)
\right)  dt+\left(  u\left(  \tau,\omega\right)  ,\phi\left(  \tau\right)
\right)  ,
\end{split}
\label{limit}%
\end{equation}

\begin{equation}%
\begin{split}
&  \int_{\tau}^{\tau+T}\left(  -\left(  u_{n}\left(  t,\omega\right)
,\phi^{\prime}\left(  t\right)  \right)  -a(\Vert u_{n}(t,\omega\Vert
_{H_{0}^{1}(\mathcal{O})}^{2})\left\langle \Delta u_{n}\left(  t,\omega
\right)  ,\phi\left(  t\right)  \right\rangle _{V^{\ast},V}\right)  dt\\
&  =\int_{\tau}^{\tau+T}\left(  f(u_{n}(t,\omega))+h(t),\phi\left(  t\right)
\right)  dt+\left(  u_{n}\left(  \tau,\omega\right)  ,\phi\left(  \tau\right)
\right)  .
\end{split}
\label{aprox}%
\end{equation}
Passing to the limit in (\ref{aprox}), taking in to account (\ref{limit}) and
bearing in mind $u_{n}(\tau,\omega)=P_{n}u_{\tau}(\omega)\rightarrow u_{\tau
}(\omega)$ we get%
\[
(u\left(  \tau,\omega),\phi\left(  \tau\right)  \right)  =\left(  u_{\tau
}(\omega),\phi\left(  \tau\right)  \right)  .
\]
Since $\phi\left(  \tau\right)  \in V\cap L^{p}(\mathcal{O})$ is arbitrary, we
infer that $u(\tau,\omega)=u_{\tau}(\omega)$.

In a similar way we check that
\begin{equation}
u(t_{0},\omega)=v.\label{initime}%
\end{equation}
By (\ref{cont0}) and (\ref{initime}) we get
\begin{equation}
u_{n}(t_{0},\omega)\rightharpoonup u(t_{0},\omega)\text{ in }%
H\label{rightconv}%
\end{equation}

Hence, $u(t,\omega)$ is a regular solution to (\ref{RD}) satisfying $u\left(
\tau,\omega\right)  =u_{\tau}(\omega),$ for a fixed $\omega$. This solution is
unique (which is proved exactly as in Theorem 13 in \cite{CMRV}), so any
converging subsequence has the same limit. Hence, by (\ref{rightconv}) the
whole sequence $u_{n}(t,\omega)$ converges weakly to $u(t,\omega)$ in $H$ for
any $t\geq\tau$ and $\omega\in\Omega$. Since $u_{n}(t,\omega)$ is measurable
in $\omega\in\Omega$, the weak limit $u(t,\omega)$ is weakly measurable, and
this implies that $\omega\mapsto u(t,\omega)$ is measurable as by the Pettis
theorem strong and weak measurability are equivalent properties when the space
is separable (see \cite[p. 42]{diestel}).

On the other hand, by (\ref{1.10.1}), (\ref{Inequn2}) and (\ref{Inequn3}) we
obtain that%
\begin{equation}
\Vert u(t,\omega)\Vert^{2}\leq\Vert u_{\tau}(\omega)\Vert^{2}+2TK_{1}%
+2K_{2}(T),\text{ }\forall\ t\in\lbrack\tau,\tau+T],\label{Inequ}%
\end{equation}%
\begin{equation}
\int_{\tau}^{\tau+T}\Vert u(s,\omega)\Vert_{V}^{2}ds\leq\frac{1}{m}\Vert
u_{\tau}(\omega)\Vert^{2}+\frac{2TK_{1}+2K_{2}(T)}{m},\label{Inequ2}%
\end{equation}%
\begin{equation}
\int_{\tau}^{\tau+T}\Vert u(s,\omega)\Vert_{L^{p}(\mathcal{O})}^{p}ds\leq
\frac{1}{2\alpha}\Vert u_{\tau}(\omega)\Vert^{2}+\frac{TK_{1}+K_{2}(T)}%
{\alpha},\label{Inequ3}%
\end{equation}%
\begin{equation}
\Vert u(t,\omega)\Vert_{V}^{2}\leq\left(  \frac{2TK_{1}+2K_{2}(T)+\Vert
u_{\tau}(\omega)\Vert^{2}}{mr}+K_{3}(T)\right)  e^{2\eta r},\ \forall
t\geq\tau+r,\label{Inequ4}%
\end{equation}%
\begin{align}
\int_{\tau+r}^{T}\Vert\Delta u_{n}(t,\omega)\Vert^{2}dt &  \leq\frac
{K_{4}(T,r)}{m}(1+\Vert u_{\tau}(\omega)\Vert^{2})+\frac{1}{2m^{2}}\int%
_{\tau+r}^{T}\Vert h(t)\Vert^{2}dt\label{Inequ5}\\
&  +\frac{\eta}{m^{2}}\left(  \Vert u_{\tau}(\omega)\Vert^{2}+2TK_{1}%
+2K_{2}(T)\right)  ,\nonumber
\end{align}
for any $r>0$, $\tau\in\mathbb{R},\ T>0$ and $\omega\in\Omega$. Since
$u_{\tau}\in L^{2}(\Omega,H),$ we have
\begin{equation}%
\begin{split}
u &  \in L_{loc}^{\infty}(\tau,\infty;L^{2}(\Omega,H))\cap L_{loc}^{2}%
(\tau,\infty;L^{2}(\Omega,V))\cap L_{loc}^{p}(\tau,\infty;L^{p}(\Omega
,L^{p}(\mathcal{O})))\\
&  \cap L^{\infty}(\tau+\varepsilon,\tau+T;L^{2}(\Omega,V))\cap L^{2}%
(\tau+\varepsilon,\tau+T;L^{2}(\Omega,D(A))),
\end{split}
\label{usol}%
\end{equation}
for any $0<\varepsilon<T$. For every fixed $\omega$, $u(\cdot,\omega)\in
C([\tau,\tau+T],H)$ and in view of (\ref{Inequ}) and the Lebesgue dominated
convergence theorem, we have that
\[
u\in C([\tau,\infty),L^{2}(\Omega,H))
\]
obtaining that $u$ is a solution in the sense of Definition \ref{defsol}.

We will prove that the solution is unique. If $u,v$ are two solutions, then
the difference $w=u-v$ satisfies%
\[
\frac{dw}{dt}-a(\Vert u\Vert_{V}^{2})\Delta u+a(\Vert v\Vert_{V}^{2})\Delta
v=f(u_{1})-f(u_{2}),
\]
so multiplying by $v$ and using (\ref{Deriv}) we have%
\[
\frac{1}{2}\frac{d}{dt}\left\Vert v\right\Vert ^{2}+\int_{\mathcal{O}}\left(
-a(\Vert u\Vert_{V}^{2})\Delta u+a(\Vert v\Vert_{V}^{2})\Delta v\right)
wdx\leq\eta\left\Vert v\right\Vert ^{2}.
\]
Since (\ref{Cond2a}) implies%
\begin{align}
&  \int_{\mathcal{O}}\left(  -a(\Vert u\Vert_{V}^{2})\Delta u+a(\Vert
v\Vert_{V}^{2})\Delta v\right)  (u-v)dx\nonumber\\
&  =\int_{\mathcal{O}}(a(\Vert u\left(  t\right)  \Vert_{V}^{2})|\nabla
u|^{2}-a(\Vert u\left(  t\right)  \Vert_{V}^{2})\nabla u\nabla v-a(\Vert
v\left(  t\right)  \Vert_{V}^{2})\nabla u\nabla v+a(\Vert v\left(  t\right)
\Vert_{V}^{2})|\nabla v|^{2})dx\nonumber\\
&  \geq a(\Vert u\left(  t\right)  \Vert_{V}^{2})\Vert u\left(  t\right)
\Vert_{V}^{2}-\left(  a(\Vert u\left(  t\right)  \Vert_{V}^{2})+a(\Vert
v\left(  t\right)  \Vert_{V}^{2})\right)  \Vert u\left(  t\right)  \Vert
_{V}\Vert v\left(  t\right)  \Vert_{V}+a(\Vert v\left(  t\right)  \Vert
_{V}^{2})\Vert v\left(  t\right)  \Vert_{V}^{2}\nonumber\\
&  =\left(  a(\Vert u\left(  t\right)  \Vert_{V}^{2})\Vert u\left(  t\right)
\Vert_{V}-a(\Vert v\left(  t\right)  \Vert_{V}^{2})\Vert v\left(  t\right)
\Vert_{V}\right)  \left(  \Vert u\left(  t\right)  \Vert_{V}-\Vert v\left(
t\right)  \Vert_{V}\right)  \geq0,\label{MonotoneA}%
\end{align}
we obtain by the Gronwall Lemma and taking expectations that%
\[
\mathbb{E}\left(  \left\Vert w(t)\right\Vert ^{2}\right)  \leq e^{2\eta
(t-\tau)}\mathbb{E}\left(  \left\Vert u(\tau)-v(\tau)\right\Vert ^{2}\right)
,
\]
which implies uniqueness and the continuity of solutions with respecto to the
initial datum as well.

Finally, equality (\ref{Energy}) is obtained multiplying the equation by
$u(t)$ and taking expectations.
\end{proof}

\bigskip

Let $\Phi$ be the mapping from $\mathbb{R}^{+}\times\mathbb{R}\times
L^{2}(\Omega,H)$ to $L^{2}(\Omega,H)$ given by
\[
\Phi(t,\tau,u_{\tau})=u(t+\tau)
\]
where $t\geq0,$ $\tau\in\mathbb{R},$ $u_{\tau}\in L^{2}(\Omega,H),$ and $u$ is
the unique solution to (\ref{RD}) with $u\left(  \tau\right)  =u_{\tau}$. By
the previous result, $\Phi$ is a continuous mean random dynamical system on
$L^{2}(\Omega,H)$.

We recall that $\lambda_{1}>0$ is the first eigenvalue of $-\Delta$ with
homogeneous Dirichlet boundary conditions, and assume that there exists
$\mu\in(0,2m\lambda_{1})$ such that
\begin{equation}
\int_{-\infty}^{0}e^{\mu s}\Vert h(s)\Vert^{2}ds<\infty.\label{3.42}%
\end{equation}

\begin{remark}
Following Remark \ref{f}, the function $\overline{h}(t)=h(t)+f(0)$ satisfies
assumption (\ref{3.42}) as well.
\end{remark}

According to the previous assumption, let us consider the following universe:
denote $\mathcal{D}$ the class of all families of nonempty bounded subsets of
$L^{2}(\Omega,H),$ $D=\{D(\tau):\tau\in\mathbb{R}\},$ such that
\begin{equation}
\lim_{\tau\rightarrow-\infty}e^{\mu\tau}\sup_{v\in D(\tau)}\Vert v\Vert
_{L^{2}(\Omega,H)}^{2}=0.\label{3.40}%
\end{equation}
Within this setting we can derive uniform estimates on the solutions to
(\ref{RD}) that will lead to the existence of $\mathcal{D}$-pullback absorbing
family in $L^{2}(\Omega,H).$ Namely, we have the following result.

\begin{lemma}
Assume that (\ref{Diss})-(\ref{Cond2a}) and (\ref{3.42}) hold. Then for any
$\tau\in\mathbb{R}$ and $D=\{D(t):t\in\mathbb{R}\}\in\mathcal{D}$ there exists
$T=T(\tau,D)>0$ such that for all $t\geq T$ and $u_{\tau-t}\in D(\tau-t)$ we
have
\[
E(\Vert u(\tau)\Vert^{2})\leq M+Me^{-\mu\tau}\int_{-\infty}^{\tau}e^{\mu
s}\Vert h(s)\Vert^{2}ds,
\]
where $u$ is the unique solution satisfying $u(t-\tau)=u_{t-\tau}$ and $M$
denotes a positive constant independent of $\tau$ and $D$ (but dependent on
$\mu$).
\end{lemma}

\begin{proof}
From the energy equality (\ref{Energy}) and assumptions (\ref{Diss}),
(\ref{Growth}), (\ref{Cond1a}) and the Poincar\'{e} inequality we have
\begin{align*}
&  \frac{d}{ds}E(\Vert u(s)\Vert^{2})+2m\lambda_{1}E(\Vert u(s)\Vert^{2})\\
\leq &  -2\alpha E(\Vert u(s)\Vert_{L^{p}(\mathcal{O})}^{p})+2\beta
|\mathcal{O}|+\frac{1}{2m\lambda_{1}-\mu}\Vert h(s)\Vert^{2}+(2m\lambda
_{1}-\mu)E(\Vert u(s)\Vert^{2})\ \text{for }a.a.\ s>\tau.
\end{align*}

Multiplying by $e^{\mu s}$ we deduce
\[
\frac{d}{ds}(e^{\mu s}E(\Vert u(s)\Vert^{2}))+2\alpha e^{\mu s}E(\Vert
u(s)\Vert_{L^{p}(\mathcal{O})}^{p})\leq2\beta|\mathcal{O}|e^{\mu s}+\frac
{1}{2m\lambda_{1}-\mu}e^{\mu s}\Vert h(t)\Vert^{2}\ a.a.\ s>\tau.
\]
Integrating in $[\tau-t,\tau]$
\begin{align*}
&  E(\Vert u(\tau)\Vert^{2})+2\alpha e^{-\mu\tau}\int_{\tau-t}^{\tau}e^{\mu
s}E(\Vert u(s)\Vert_{L^{p}(\mathcal{O})}^{p})ds\\
\leq &  e^{-\mu\tau}e^{\mu(\tau-t)}E(\Vert u_{\tau-t}\Vert^{2})+\frac
{1}{2m\lambda_{1}-\mu}e^{-\mu\tau}\int_{\tau-t}^{\tau}e^{\mu s}\Vert
h(s)\Vert^{2}ds+2\frac{\beta|\mathcal{O}|}{\mu}\quad\forall t\geq0.
\end{align*}
Since $u_{\tau-t}\in D(\tau-t),$ we have that there exists $T=T(\tau,D)$ such
that
\[
e^{-\mu\tau}e^{\mu(\tau-t)}E(\Vert u_{\tau-t}\Vert^{2})\leq1\quad\forall t\geq
T.
\]
The proof is complete.
\end{proof}

\begin{corollary}
Suppose that (\ref{Diss})-(\ref{Cond2a}) and (\ref{3.42}) hold. Then, the
family $K=\{K(\tau):\tau\in\mathbb{R}\}$ with $K(\tau)=\{u\in L^{2}%
(\Omega,L^{2}(\mathcal{O})):E(\Vert u\Vert^{2})\leq R(\tau)\},$ where
\[
R(\tau)=M+Me^{-\mu\tau}\int_{-\infty}^{\tau}e^{\mu s}\Vert h(s)\Vert^{2}ds,
\]
belongs to $\mathcal{D}$ and is a weakly compact $\mathcal{D}$-pullback
absorbing family for $\Phi.$
\end{corollary}

This allows us to use Theorem \ref{AttrExist} to conclude the main result of
this section.

\begin{theorem}
Suppose that (\ref{Diss})-(\ref{Cond2a}) and (\ref{3.42}) hold. Then, the
continuous mean random dynamical system $\Phi$ defined through the solutions
to problem (\ref{RD}) has the unique weak $\mathcal{D}$-pullback mean random
attractor $\mathcal{A}=\{\mathcal{A}(\tau):\tau\in\mathbb{R}\}\in\mathcal{D}.$
\end{theorem}

In general, the radius $R(\tau)$ can be unbounded as $\tau\rightarrow\pm
\infty$. However, under an additional assumption on the function $h\left(
t\right)  $ we are able to force it to be bounded in either one or both
directions. The following result is straightforward to check.

\begin{lemma}
\label{AttrBoundedRandom}Suppose that (\ref{Diss})-(\ref{Cond2a}) and
(\ref{3.42}) hold. If, additionally,
\[
\sup_{t\leq t_{0}}e^{-\mu t}\int_{-\infty}^{t}e^{\mu r}\left\Vert
h(r)\right\Vert ^{2}dr<\infty
\]
for some $t_{0}\in\mathbb{R}$, then $\sup_{\tau\leq t}R(\tau)<\infty$ for any
$t\in\mathbb{R}$. Hence, the union $\cup_{\tau\leq t}\mathcal{A}(\tau)$ is
bounded for any $t\in\mathbb{R}$.

If
\[
\sup_{t\in\mathbb{R}}e^{-\mu t}\int_{-\infty}^{t}e^{\mu r}\left\Vert
h(r)\right\Vert ^{2}dr<\infty,
\]
then $\sup_{\tau\in\mathbb{R}}R(\tau)<\infty$. Hence, the union $\cup_{\tau
\in\mathbb{R}}\mathcal{A}(\tau)$ is bounded.
\end{lemma}

\begin{corollary}
If $h$ does not depend on time, that is, $h(t)\equiv h_{0}\in H$, then the
union $\cup_{\tau\in\mathbb{R}}\mathcal{A}(\tau)$ is bounded.
\end{corollary}

\section{The mean random attractor for a stochastic non-local problem}

We consider now the following stochastic nonlocal reaction-diffusion equation
\begin{equation}
\left\{
\begin{array}
[c]{l}%
du=(a(\Vert u\Vert_{V}^{2})\Delta u+f(u)+h(t,x))dt+\sigma\left(  u\right)
dw\left(  t\right)  \ \text{in }(\tau,\infty)\times\mathcal{O},\\
u=0\quad\text{on }(\tau,\infty)\times\partial\mathcal{O},\\
u(\tau,x)=u_{\tau}(x)\quad\text{for }x\in\mathcal{O},
\end{array}
\right.  \label{1}%
\end{equation}
where $\mathcal{O}$ is a bounded open set of $\mathbb{R}^{n}$ with smooth
boundary $\partial\mathcal{O},$ $w\left(  t\right)  $ is a two-sided scalar
Wiener process with respect to the filtration $\{\mathcal{F}_{t}%
\}_{t\in\mathbb{R}}$ and $(\Omega,\mathcal{F},\{\mathcal{F}_{t}\}_{t\in
\mathbb{R}},P)$ is a complete filtered probability space such that
$\{\mathcal{F}_{t}\}_{t\in\mathbb{R}}$ is a right continuous family of
sub-$\sigma$-algebras of $\mathcal{F}$ that contains all $P$-null sets. The
integral is understood in the It\^{o} sense.

The functions $f\in C^{1}(\mathbb{R})$, $a\in C(\mathbb{R}^{+}),$
$\sigma:\mathbb{R}\rightarrow\mathbb{R}$ and $h$ satisfy:%
\begin{equation}
f^{\prime}\left(  r\right)  \leq\gamma_{1},\label{f1}%
\end{equation}%
\begin{equation}
\left\vert f\left(  r\right)  \right\vert \leq\gamma_{2}\left(  1+\left\vert
r\right\vert \right)  ,\label{f2}%
\end{equation}%
\begin{equation}
f\left(  r\right)  r\leq\gamma_{3}+\gamma_{4}r^{2},\label{f3}%
\end{equation}%
\begin{equation}
0<m\leq a\left(  s\right)  \leq M,\ \forall s\geq0,\label{a1}%
\end{equation}%
\begin{equation}
s\mapsto a(s^{2})s\text{ is non-decreasing,}\label{a2}%
\end{equation}%
\begin{equation}
h\in L^{2}(\tau,T;H),\text{ for all }\tau<T,\label{h}%
\end{equation}%
\begin{equation}
\sigma\text{ is globally Lipschitz (with constant }C_{\sigma}\text{),}%
\label{sigma}%
\end{equation}
for some $\gamma_{i},m,M>0$ and all $s\geq0,\ r\in\mathbb{R}.$ Additionally,
we will need to assume that
\begin{equation}
\gamma_{4}+C_{\sigma}^{2}<\frac{m\lambda_{1}}{2}.\label{CondDiss}%
\end{equation}
Although some of these conditions are the same as in Section
\ref{AttractorRandom}, for clarity of exposition we prefer to write them here again.

Let $\widetilde{f}:H\rightarrow H$ be the Nemitsky operator given by
$\widetilde{f}\left(  u\right)  \left(  x\right)  =f\left(  u\left(  x\right)
\right)  $ for almost all $x\in\mathcal{O}$. We define the operator
$B:\mathbb{R}\times V\rightarrow V^{\ast}$ given by%
\[
\left\langle B\left(  t,u\right)  ,v\right\rangle _{V^{\ast},V}=-a(\Vert
u\Vert_{V}^{2})\left(  \nabla u,\nabla v\right)  +\left(  \widetilde{f}\left(
u\right)  +h(t),v\right)  .
\]
It is straigthforward to see using Lebesgue's theorem that the operator
$\widetilde{f}$ is continuous.

In the same way, let $\widetilde{\sigma}:H\rightarrow H$ be the Nemitsky
operator given by $\widetilde{\sigma}\left(  u\right)  \left(  x\right)
=\sigma\left(  u\left(  x\right)  \right)  $ for almost all $x\in\mathcal{O}$.
It is clear that $\widetilde{\sigma}$ is globally Lipschitz as well. Indeed,%
\begin{align}
\left\Vert \widetilde{\sigma}(u)-\widetilde{\sigma}(v)\right\Vert  &  =\left(
\int_{\mathcal{O}}(\sigma(u(x))-\sigma(v(x)))^{2}dx\right)  ^{\frac{1}{2}%
}\nonumber\\
&  \leq C_{\sigma}\left(  \int_{\mathcal{O}}(u(x)-v(x))^{2}dx\right)
^{\frac{1}{2}}=C_{\sigma}\left\Vert u-v\right\Vert .\label{Lipschitz}%
\end{align}

Under conditions (\ref{f1})-(\ref{sigma}) the following lemmas hold.

\begin{lemma}
\label{hemicont}$B$ is hemicontinuous.
\end{lemma}

\begin{proof}
Since $a,\widetilde{f}$ are continuous, for any $u,v,z\in V$ we have that the
function%
\[
\lambda\mapsto\left\langle B\left(  t,u+\lambda z\right)  ,v\right\rangle
_{V^{\ast},V}=a(\Vert u+\lambda z\Vert_{V}^{2})\left(  \left(  \nabla u,\nabla
v\right)  +\lambda\left(  \nabla z,\nabla v\right)  \right)  +\left(
\widetilde{f}\left(  u+\lambda z\right)  +h(t),v\right)
\]
is continuous. Hence, $B$ is hemicontinuous.
\end{proof}

\bigskip

\begin{lemma}
\label{monotone}$B$ is weakly monotone, that is, there exists $c>0$ such that%
\[
2\left\langle B\left(  t,u\right)  -B(t,v),u-v\right\rangle _{V^{\ast}%
,V}+\left\Vert \sigma\left(  u\right)  -\sigma\left(  v\right)  \right\Vert
_{H}^{2}\leq c\left\Vert u-v\right\Vert ^{2}\text{ }\forall u,v\in V.
\]

\end{lemma}

\begin{proof}
By (\ref{f1}) we obtain that%
\begin{align*}
\left(  \widetilde{f}\left(  u\right)  -\widetilde{f}\left(  v\right)
,u-v\right)   &  =\int_{\mathcal{O}}f^{\prime}\left(  \alpha\left(  x\right)
u\left(  x\right)  +\left(  1-\alpha\left(  x\right)  \right)  v\left(
x\right)  \right)  \left(  u\left(  x\right)  -v\left(  x\right)  \right)
^{2}dx\\
&  \leq\gamma_{1}\left\Vert u-v\right\Vert ^{2}.
\end{align*}
Hence, the result follows from (\ref{MonotoneA}) and (\ref{Lipschitz}).
\end{proof}

\bigskip

\begin{lemma}
\label{coercive}$B$ is coercive, that is, there are $c_{1}\geq0,\ c_{2}>0$
such that%
\[
2\left\langle B\left(  t,u\right)  ,u\right\rangle _{V^{\ast},V}+\left\Vert
\widetilde{\sigma}\left(  u\right)  \right\Vert _{H}^{2}\leq c_{1}\left\Vert
u\right\Vert ^{2}-c_{2}\left\Vert u\right\Vert _{V}^{2}+g(t),
\]
where $g\in L^{1}(\tau,T)$ for all $\tau<T.$
\end{lemma}

\begin{proof}
In view of (\ref{a1}), (\ref{f3}) and (\ref{sigma}) we have the inequalities%
\[
-a(\Vert u\Vert_{V}^{2})\left(  \nabla u,\nabla u\right)  \leq-m\left\Vert
u\right\Vert _{V}^{2},
\]%
\[
\left(  \widetilde{f}\left(  u\right)  ,u\right)  \leq\gamma_{3}\left\vert
\mathcal{O}\right\vert +\gamma_{4}\left\Vert u\right\Vert ^{2},
\]%
\[
\left\Vert \widetilde{\sigma}\left(  u\right)  \right\Vert _{H}^{2}\leq\left(
\left\Vert \widetilde{\sigma}\left(  0\right)  \right\Vert _{H}+C_{\sigma
}\left\Vert u\right\Vert \right)  ^{2},
\]%
\[
(h(t),u)\leq\left\Vert h(t)\right\Vert \left\Vert u\right\Vert \leq\frac{1}%
{2}\left\Vert h(t)\right\Vert ^{2}+\frac{1}{2}\left\Vert u\right\Vert ^{2},
\]
which imply the result by putting $g(t)=2\gamma_{3}\left\vert \mathcal{O}%
\right\vert +2\left\Vert \widetilde{\sigma}\left(  0\right)  \right\Vert
_{H}^{2}+\left\Vert h(t)\right\Vert ^{2}$, $c_{1}=2\gamma_{4}+1+2C_{\sigma}$,
$c_{2}=2m$.
\end{proof}

\bigskip

\begin{lemma}
\label{Abounded}The operator $B$ is bounded, that is, there are $d_{1}%
,d_{2}\geq0$ such that%
\[
\left\Vert B(t,u)\right\Vert _{V^{\ast}}\leq d_{1}+d_{2}\left\Vert
u\right\Vert _{V}+g(t)\text{ }\forall u\in V,
\]
where $g\in L^{2}(\tau,T)$ for all $\tau<T.$
\end{lemma}

\begin{proof}
It follows from (\ref{a1}) and (\ref{f2}) that%
\begin{align*}
\left\Vert B(t,u)\right\Vert _{V^{\ast}} &  \leq M\left\Vert u\right\Vert
_{V}+\left\Vert \widetilde{f}\left(  u\right)  +h(t)\right\Vert _{V^{\ast}}+\\
&  \leq M\left\Vert u\right\Vert _{V}+\widetilde{c}\left(  \left\Vert
\widetilde{f}\left(  u\right)  \right\Vert +\left\Vert h(t)\right\Vert
\right)  \\
&  \leq M\left\Vert u\right\Vert _{V}+\widetilde{c}\gamma_{2}\sqrt{2\left\vert
\mathcal{O}\right\vert +2\left\Vert u\right\Vert ^{2}}+\widetilde{c}\left\Vert
h(t)\right\Vert \\
&  \leq d_{1}+d_{2}\left\Vert u\right\Vert _{V}+g(t).
\end{align*}

\end{proof}

\bigskip

We will focus first on the existence and uniqueness of solutions to problem
(\ref{1}).

\begin{definition}
For $\tau\in\mathbb{R}$ and $u_{\tau}\in L^{2}(\Omega,\mathcal{F}_{\tau};H),$
an $H$-valued $\{\mathcal{F}_{t}\}_{t\in\mathbb{R}}$-adapted stochastic
process $u$ is called a solution to problem (\ref{1}) if $u\in C([\tau
,+\infty),H)\cap L_{loc}^{2}(\tau,+\infty;V)$ P-a.s. and satisfies the
equality%
\begin{align*}
&  (u(t),\xi)+\int_{\tau}^{t}a(\left\Vert u\right\Vert _{V}^{2})(\nabla
u,\nabla\xi)ds\\
&  =(u_{\tau},\xi)+\int_{\tau}^{t}(\widetilde{f}(u(s))+h(s),\xi)ds+\int_{\tau
}^{t}(\widetilde{\sigma}(u(s)),\xi)dw(s),\ \forall\ \tau<t,
\end{align*}
P-almost sure for all $\xi\in V.$
\end{definition}

From Lemmas \ref{hemicont}-\ref{Abounded} and Theorems 4.2.4 and 4.2.5 in
\cite{PrevotRockner} we obtain that a unique solution exists and that it
satisfies the It\^{o} formula.

\begin{lemma}
Assume (\ref{f1})-(\ref{sigma}). Then for any $u_{\tau}\in L^{2}%
(\Omega,\mathcal{F}_{\tau};H)$ there exists a unique solution $u$ to problem
(\ref{1}) which, moreover, satisfies%
\begin{equation}
\mathbb{E}\left(  \sup_{t\in\lbrack\tau,\tau+T]}\left\Vert u(t)\right\Vert
^{2}\right)  <\infty,\label{ExpectBounded}%
\end{equation}%
\begin{align}
&  \left\Vert u(t)\right\Vert ^{2}+2\int_{\tau}^{t}a(\left\Vert u\right\Vert
_{V}^{2})\left\Vert \nabla u\right\Vert ^{2}ds\label{Ito}\\
&  =\left\Vert u(\tau)\right\Vert ^{2}+2\int_{\tau}^{t}\left(  \widetilde{f}%
(u(s))+h(s),u(s)\right)  ds+\int_{\tau}^{t}\left\Vert \widetilde{\sigma
}(u(s))\right\Vert ^{2}ds+2\int_{\tau}^{t}(\widetilde{\sigma}%
(u(s)),u(s))dw(s)\ \text{P-a.s., }\nonumber
\end{align}
for all $\tau<T.$
\end{lemma}

By using (\ref{ExpectBounded}), $u\in C([\tau,+\infty),H)$ P-a.s. and the
Lebesgue theorem we obtain that $u\in C([\tau,\infty),L^{2}(\Omega,H))$, so we
define the map $\Phi:\mathbb{R}^{+}\times\mathbb{R}\times L^{2}(\Omega
,H)\rightarrow L^{2}(\Omega,H)$ by
\[
\Phi(t,\tau,u_{\tau})=u(t+\tau),
\]
where $u$ is the unique solution to (\ref{1}) with $u(\tau)=u_{\tau}$. By the
uniqueness of solutions, this family of mappings is a mean random dynamical
system on $L^{2}(\Omega,\mathcal{F};H)$ over $(\Omega,\mathcal{F}%
,\{\mathcal{F}_{t}\}_{t\in\mathbb{R}},\mathbb{P})$ in the sense of Definition
\ref{MRDS2}.

We start with an a priori estimate.

\begin{lemma}
\label{Est1}Assume (\ref{f1})-(\ref{CondDiss}). Then there are constants
$K_{1},K_{2}>0$ such that for any $u_{\tau}\in L^{2}(\Omega,\mathcal{F}_{\tau
};H)$ the solution $u$ satisfies the estimate%
\[
\mathbb{E}(\left\Vert u(t)\right\Vert ^{2})\leq e^{-\omega_{0}(t-\tau
)}\mathbb{E}(\left\Vert u_{\tau}\right\Vert ^{2})+\frac{K_{1}}{\omega_{0}%
}+K_{2}\int_{\tau}^{t}e^{-\omega_{0}(t-r)}\left\Vert h(r)\right\Vert
^{2}dr\text{, for all }\tau<t,
\]
for any $0<\omega_{0}<m\lambda_{1}-\gamma_{4}-C_{\sigma}^{2}$.
\end{lemma}

\begin{proof}
Taking expectations in (\ref{Ito}) we have%
\begin{align*}
&  \mathbb{E}\left(  \left\Vert u(r)\right\Vert ^{2}\right)  +2\int_{\tau}%
^{r}\mathbb{E}\left(  a(\left\Vert \nabla u\right\Vert ^{2})\left\Vert \nabla
u\right\Vert ^{2}\right)  ds\\
&  =\mathbb{E}\left(  \left\Vert u(\tau)\right\Vert ^{2}\right)  +2\int_{\tau
}^{r}\mathbb{E}\left(  \widetilde{f}(u(s))+h(s),u(s)\right)  ds+\int_{\tau
}^{r}\mathbb{E}\left(  \left\Vert \widetilde{\sigma}(u(s))\right\Vert
^{2}\right)  ds\ \text{ for }r\geq\tau.
\end{align*}
Thus, for a.a. $r>\tau,$%
\[
\frac{d}{dr}\mathbb{E}\left(  \left\Vert u(r)\right\Vert ^{2}\right)
+2\mathbb{E}\left(  a(\left\Vert \nabla u\right\Vert ^{2})\left\Vert \nabla
u(r)\right\Vert ^{2}\right)  =2\mathbb{E}\left(  \widetilde{f}%
(u(r))+h(r),u(r)\right)  +\mathbb{E}\left(  \left\Vert \widetilde{\sigma
}(u(r))\right\Vert ^{2}\right)  .
\]

We estimate each term by using (\ref{f3}), (\ref{h}) and (\ref{sigma}):%
\[
\left(  \widetilde{f}(u(r)),u(r)\right)  \leq\gamma_{3}\left\vert
\mathcal{O}\right\vert +\gamma_{4}\left\Vert u(r)\right\Vert ^{2},
\]%
\[
\left(  h(r),u(r)\right)  \leq\varepsilon\left(  m\lambda_{1}-\gamma
_{4}-C_{\sigma}^{2}\right)  \left\Vert u(r)\right\Vert ^{2}+\frac
{1}{4\varepsilon\left(  m\lambda_{1}-\gamma_{4}-C_{\sigma}^{2}\right)
}\left\Vert h(r)\right\Vert ^{2},
\]%
\begin{align*}
\left\Vert \widetilde{\sigma}(u(s))\right\Vert ^{2} &  \leq\left(  \left\Vert
\widetilde{\sigma}\left(  0\right)  \right\Vert _{H}+C_{\sigma}\left\Vert
u\right\Vert \right)  ^{2}\\
&  \leq2\left\Vert \widetilde{\sigma}\left(  0\right)  \right\Vert _{H}%
^{2}+2C_{\sigma}^{2}\left\Vert u\right\Vert ^{2},
\end{align*}
where $\varepsilon\in(0,1)$. Hence,%
\begin{align*}
&  \frac{d}{dr}\mathbb{E}\left(  \left\Vert u(r)\right\Vert ^{2}\right)
+2(1-\varepsilon)(m\lambda_{1}-\gamma_{4}-C_{\sigma}^{2})\mathbb{E}\left(
\left\Vert u(r)\right\Vert ^{2}\right)  \\
&  \leq2C_{\sigma}^{2}\left\Vert \widetilde{\sigma}\left(  0\right)
\right\Vert _{H}^{2}+\frac{1}{2\varepsilon\left(  m\lambda_{1}-2\gamma
_{4}-2C_{\sigma}^{2}\right)  }\left\Vert h(r)\right\Vert ^{2}+2\gamma
_{3}\left\vert \mathcal{O}\right\vert \\
&  =K_{1}+K_{2}\left\Vert h(r)\right\Vert ^{2}.
\end{align*}
Thus, by the Gronwall lemma,%
\[
\mathbb{E}(\left\Vert u(t)\right\Vert ^{2})\leq e^{-\omega_{0}(t-\tau
)}\mathbb{E}(\left\Vert u_{\tau}\right\Vert ^{2})+\frac{K_{1}}{\omega_{0}%
}+K_{2}\int_{\tau}^{t}e^{-\omega_{0}(t-r)}\left\Vert h(r)\right\Vert ^{2}dr,
\]
for any $0<\omega_{0}<m\lambda_{1}-\gamma_{4}-C_{\sigma}^{2}$.
\end{proof}

\bigskip

Further, let us consider the following condition: for some $0<\omega
_{0}<m\lambda_{1}-\gamma_{4}-C_{\sigma}^{2}$ the function $h$ satisfies that
\begin{equation}
\int_{-\infty}^{t}e^{\omega_{0}r}\left\Vert h(r)\right\Vert ^{2}%
dr<\infty\text{ for all }t\in\mathbb{R}.\label{IntCond}%
\end{equation}

Fixing the constant $\omega_{0}$ from (\ref{IntCond}) we denote by
$\mathcal{D}$ the collection of all families of non-empty bounded subsets
$D=\{D(\tau):\tau\in\mathbb{R}\}$, $D(\tau)\subset L^{2}(\Omega,\mathcal{F}%
_{\tau};H)$, such that%
\[
\lim_{\tau\rightarrow-\infty}e^{\omega_{0}\tau}\left\Vert D(\tau)\right\Vert
_{+}^{2}=0,
\]
where $\left\Vert D(\tau)\right\Vert _{+}=\sup_{y\in D(\tau)}\left\Vert
y\right\Vert _{L^{2}(\Omega,\mathcal{F}_{\tau};H)}.$

\begin{lemma}
\label{Est2}Assume (\ref{f1})-(\ref{CondDiss}) and also condition
(\ref{IntCond}). Then for any $t\in\mathbb{R}$ and $D=\{D(\tau):\tau
\in\mathbb{R}\}\in\mathcal{D}$ there exists $T=T(t,D)$ such that if $s\geq T$,
then every solution $u$ with initial condition at time $\tau=t-s$ given by
$u_{t-s}\in D(t-s)$ satisfies
\begin{equation}
\mathbb{E}(\left\Vert u(t)\right\Vert ^{2})\leq M(1+e^{-\omega_{0}t}%
\int_{-\infty}^{t}e^{\omega_{0}r}\left\Vert h(r)\right\Vert ^{2}%
dr)=:R_{0}(t)\text{, } \label{R0}%
\end{equation}
where $M>0$ is a constant which depends on $\omega_{0}.$
\end{lemma}

\begin{proof}
Since $u_{t-s}\in D(t-s)$, we have%
\[
e^{-\omega_{0}s}\mathbb{E}(\left\Vert u_{t-s}\right\Vert ^{2})\leq
e^{-\omega_{0}t}e^{\omega_{0}(t-s)}\left\Vert D(t-s)\right\Vert _{+}%
^{2}\rightarrow0\text{ as }s\rightarrow+\infty,
\]
so there exists $T(t,D)$ for which
\[
e^{-\omega_{0}s}\mathbb{E}(\left\Vert u_{t-s}\right\Vert ^{2})\leq1\text{ if
}s\geq T.
\]
From the estimate in Lemma \ref{Est1} we obtain that%
\[
\mathbb{E}(\left\Vert u(t)\right\Vert ^{2})\leq1+\frac{K_{1}}{\omega_{0}%
}+K_{2}e^{-\omega_{0}t}\int_{-\infty}^{t}e^{\omega_{0}r}\left\Vert
h(r)\right\Vert ^{2}dr,\text{ for }s\geq T,
\]
proving the result.
\end{proof}

\bigskip

We define now the family of bounded closed convex sets $K_{0}=\{K_{0}%
(t):\tau\in\mathbb{R}\}$ given by%
\[
K_{0}(t)=\{u\in L^{2}(\Omega,\mathcal{F}_{t};H):\mathbb{E}(\left\Vert
u\right\Vert ^{2})\leq R_{0}(t)\},
\]
where $R_{0}(t)$ is the function in (\ref{R0}). It is clear that the sets
$K_{0}(t)$ are weakly compact. We will prove that this family is $\mathcal{D}%
$-pullback absorbing and that $K_{0}\in\mathcal{D}$.

\begin{lemma}
\label{Absorbing}Assume the conditions of Lemma \ref{Est2}. Then $K_{0}$ is a
weakly compact $\mathcal{D}$-pullback absorbing family which belongs to
$\mathcal{D}$.
\end{lemma}

\begin{proof}
In view of Lemma \ref{Est2}, for any $t\in\mathbb{R}$ and $D=\{D(\tau):\tau
\in\mathbb{R}\}\in\mathcal{D}$ there exists $T=T(t,D)$ such that if $s\geq T$,
then%
\[
\Phi(s,t-s,D(t-s))\subset K_{0}(t),
\]
so $K_{0}$ is a weakly compact $\mathcal{D}$-pullback absorbing family.
Finally, we see that%
\[
e^{\omega_{0}\tau}R_{0}(\tau)=M(e^{\omega_{0}\tau}+\int_{-\infty}^{\tau
}e^{\omega_{0}r}\left\Vert h(r)\right\Vert ^{2}dr)\rightarrow0\text{ as }%
\tau\rightarrow-\infty,
\]
so $K_{0}\in\mathcal{D}.$
\end{proof}

\bigskip

From Lemma \ref{Absorbing} and Theorem \ref{AttrExist2} we deduce the main
result concerning the existence of the weak $\mathcal{D}$-pullback attractor.

\begin{theorem}
Assume the conditions of Lemma \ref{Est2}. Then the mean random dynamical
system $\Phi$ has a unique weak $\mathcal{D}$-pullback mean random attractor
$\mathcal{A}_{0}=\{\mathcal{A}(\tau):\tau\in\mathbb{R}\}\in\mathcal{D}$.
\end{theorem}

\bigskip

As in the previous section, under an additional assumption on the function
$h\left(  t\right)  $ we can prove that the radius $R_{0}(t)$ and the weak
$\mathcal{D}$-pullback mean random attractor are bounded in either one or both directions.

\begin{lemma}
\label{AttrBounded}Under the conditions of Lemma \ref{Est2}, if
\[
\sup_{t\leq t_{0}}e^{-\omega_{0}t}\int_{-\infty}^{t}e^{\omega_{0}r}\left\Vert
h(r)\right\Vert ^{2}dr<\infty
\]
for some $t_{0}\in\mathbb{R}$, then $\sup_{t\leq\overline{t}}R_{0}(t)<\infty$
for any $\overline{t}\in\mathbb{R}$. Hence, the union $\cup_{t\leq\overline
{t}}\mathcal{A}(t)$ is bounded for any $\overline{t}\in\mathbb{R}$.

If
\[
\sup_{t\in\mathbb{R}}e^{-\omega_{0}t}\int_{-\infty}^{t}e^{\omega_{0}%
r}\left\Vert h(r)\right\Vert ^{2}dr<\infty,
\]
then $\sup_{t\in\mathbb{R}}R_{0}(t)<\infty$. Hence, the union $\cup
_{t\in\mathbb{R}}\mathcal{A}(t)$ is bounded.
\end{lemma}

\begin{corollary}
If $h$ does not depend on time, that is, $h(t)\equiv h_{0}\in H$, then the
union $\cup_{t\in\mathbb{R}}\mathcal{A}(t)$ is bounded.
\end{corollary}

\bigskip

\textbf{Acknowledgments.}

This work has been partially supported by the Spanish Ministry of Science,
Innovation and Universities, project PGC2018-096540-B-I00, by the Spanish
Ministry of Science and Innovation, project PID2019-108654GB-I00, and by the
Junta de Andaluc\'{\i}a and FEDER, projects US-1254251, P18-FR-4509.

This manuscript is dedicated to the memory of our colleague in the department
and research group and friend Mar\'{\i}a Jos\'{e} Garrido-Atienza who
unfortunately passed away in January 2021, on her 49th birthday. She was a
very kind, brave, warm, vital, joyful, plenty of life energy, enthusiastic
person, and a source of happiness in any meeting she participated. We will
miss her a lot and will keep her forever in our hearts. Our big hug to her
family, with deep love and sorrow.

\end{document}